\documentclass[12pt]{amsart}
\usepackage{latexsym}
\usepackage{amssymb}
\usepackage{amscd}
\renewcommand\a{\alpha}
\renewcommand\b{\beta}
\newcommand\g{\gamma}
\renewcommand\d{\delta}
\newcommand\la{\lambda}

\newcommand\e{\eta}
\renewcommand\th{\theta}
\newcommand\io{\iota}

\newcommand\s{\sigma}
\newcommand\x{\chi}
\newcommand\f{\phi}
\newcommand\vf{\varphi}

\renewcommand\t{\tau}
\renewcommand\r{\rho}

\newcommand\Om{\Omega}
\newcommand\w{\omega}

\newcommand\vD{\varDelta}
\newcommand\F{\Phi}

\newcommand\vL{\varLambda}

\newcommand\vth{\vartheta}
\newcommand\vT{\varTheta}
\newcommand\vG{\varGamma}
\newcommand\ve{\varepsilon}

\newcommand{\ZZ}{\mathbb Z}

\newcommand\Fq{{\mathbf F}_q}

\newcommand\Ql{\bar{\mathbf Q}_l}

\newcommand\BP{\mathbf P}
\newcommand\BQ{\mathbf Q}
\newcommand\BF{\mathbf F}
\newcommand\BC{\mathbf C}

\newcommand\BZ{\mathbf Z}

\newcommand\BH{\mathbf H}

\newcommand\Bm{\mathbf m}

\newcommand\Bk{\mathbf k}

\newcommand\Bla{\boldsymbol\lambda}

\newcommand\Bmu{\boldsymbol\mu}

\newcommand\CA{\mathcal{A}}
\newcommand\CB{\mathcal{B}}
\newcommand\ZC{\mathcal{C}}
\newcommand\CH{\mathcal{H}}

\newcommand\CE{\mathcal{E}}
\newcommand\CL{\mathcal{L}}
\newcommand\DD{\mathcal{D}}

\newcommand\CM{\mathcal{M}}

\newcommand\CO{\mathcal{O}}

\newcommand\CP{\mathcal{P}}

\newcommand\CU{\mathcal{U}}
\newcommand\CV{\mathcal{V}}

\newcommand\CZ{ \mathcal{Z}}
\newcommand\CX{ \mathcal{X}}
\newcommand\CY{ \mathcal{Y}}
\newcommand\CW{ \mathcal{W}}

\newcommand\Fa{\mathfrak a}

\newcommand\Fg{\mathfrak g}

\newcommand\Fl{\mathfrak l}

\newcommand\iv{^{-1}}
\newcommand\wh{\widehat}
\newcommand\wt{\widetilde}
\newcommand\wg{^{\wedge}}

\newcommand\ol{\overline}

\newcommand\hra{\hookrightarrow}

\newcommand\IC{\operatorname{IC}}

\newcommand\Hom{\operatorname{Hom}}
\newcommand\End{\operatorname{End}}
\newcommand\Aut{\operatorname{Aut}}

\newcommand\Ind{\operatorname{Ind}}

\newcommand\supp{\operatorname{supp}\,}
\newcommand\Lie{\operatorname{Lie}}
\newcommand\Tr{\operatorname{Tr}\,}
\newcommand\ch{\operatorname{ch}}

\newcommand\Ad{\operatorname{Ad}}
\newcommand\ad{\operatorname{ad}}

\newcommand\reg{_{\operatorname{reg}}}

\renewcommand\ss{\operatorname{ss}}
\newcommand\unip{\operatorname{uni}}
\newcommand\uni{_{\operatorname{uni}}}
\newcommand\nil{_{\operatorname{nil}}}
\newcommand\id{\operatorname{id}}

\newcommand\lp{\operatorname{\!\langle\!}}
\newcommand\rp{\operatorname{\!\rangle}}
\renewcommand\Im{\operatorname{Im}}

\newcommand\odd{\operatorname{odd}}

\newcommand\sym{\operatorname{sym}}
\newcommand\ex{\operatorname{ex}}
\newcommand\en{\operatorname{en}}

\newcommand\dw{\dot w}

\newcommand{\isom}{\,\raise2pt\hbox{$\underrightarrow{\sim}$}\,}
\numberwithin{equation}{section}

\newtheorem{thm}{Theorem}[section]
\newtheorem{lem}[thm]{Lemma}
\newtheorem{cor}[thm]{Corollary}
\newtheorem{prop}[thm]{Proposition}

\def \para#1{\par\medskip\textbf{#1}
              \addtocounter{thm}{1}}

\def \remark#1{\par\medskip\noindent
                \textbf{Remark #1}
                \addtocounter{thm}{1}}

\def \remarks#1{\par\medskip\noindent
                \textbf{Remarks #1}
                \addtocounter{thm}{1}}

\begin{document}
\setlength{\baselineskip}{4.9mm}
\setlength{\abovedisplayskip}{4.5mm}
\setlength{\belowdisplayskip}{4.5mm}
\renewcommand{\theenumi}{\roman{enumi}}
\renewcommand{\labelenumi}{(\theenumi)}
\renewcommand{\thefootnote}{\fnsymbol{footnote}}
\renewcommand{\thefootnote}{\fnsymbol{footnote}}
\allowdisplaybreaks[2]
\parindent=20pt
\medskip
\begin{center}
 {\bf Exotic symmetric space over a finite field, II} 
\\
\vspace{1cm}
Toshiaki Shoji\footnote{
partly supported by JSPS Grant-in Aid for Scientific Research (A) 20244001.}
 and Karine Sorlin\footnote{ supported by 
ANR JCJC REPRED, ANR-09-JCJC-0102-01.}
\\
\vspace{0.3cm}
\end{center}
\title{}

\begin{abstract}
This paper is the second part of the papers in the same title.
In this paper, we prove a conjecture of Achar-Henderson, 
which asserts that the Poincar\'e polynomials of the intersection
cohomology complex associated to the closure of $Sp_{2n}$-orbits
in the Kato's exotic nilpotent cone coincide with the modified 
Kostka polynomials indexed by double partitions, introduced by 
the first author.  Actually this conjecture was recently proved 
by Kato by a different method.  Our approach is based on the
theory of character sheaves on the exotic symmetric space.     

\end{abstract}
\maketitle
\markboth{SHOJI AND SORLIN}{ EXOTIC SYMMETRIC SPACE, II}
\pagestyle{myheadings}

\par\bigskip

\begin{center}
{\sc Introduction}
\end{center}
\par\medskip
This paper is a continuation of [SS]. 
We basically follow the notation in [SS].
In particular, $V$ is an $2n$-dimensional vector space over 
$\Bk$, an algebraic closure of a finite field $\Fq$ with 
$\ch \Bk \ne 2$, and $G = GL(V), H = G^{\th} \simeq Sp_{2n}$ 
for an involutive automorphism $\th$ on $G$.  Let 
$G^{\io\th} = \{ g \in G \mid \th(g) = g\iv \}$ (here 
$\io: G \to G, g \mapsto g\iv$) and $G^{\io\th}\uni$ be the 
set of unipotent elements in $G^{\io\th}$.   We denote 
$\CX = G^{\io\th} \times V$ and $\CX\uni = G^{\io\th}\uni \times V$, 
on which $H$ acts diagonally.
$\CX\uni$ is isomorphic to the exotic nilpotent cone introduced 
in [Ka1], and the set of $H$-orbits in $\CX\uni$ is naturally parametrized 
by the set $\CP_{n,2}$ of double partitions of $n$.
Let $\CO_{\Bla}$ be the orbit in $\CX\uni$ corresponding to 
$\Bla \in \CP_{n,2}$, and we consider the intersection cohomology 
$K = \IC(\ol\CO_{\Bla}, \Ql)$ associated to $\CO_{\Bla}$. 
Achar and Henderson conjectured in [AH] that
$\CH^jK = 0$ unless $j \equiv 0 \pmod 4$ and that 
\begin{equation*}
\tag{*}
\sum_{i \ge 0}(\dim \CH^{4i}_zK)t^{2i} = t^{-a(\Bla)}\wt K_{\Bla,\Bmu}(t)
\end{equation*}
for $z \in \CO_{\Bmu} \subset \ol\CO_{\Bla}$, 
where $\wt K_{\Bla, \Bmu}(t)$ is a modified Kostka polynomial 
associated to double partitions $\Bla, \Bmu$, introduced by 
the first author in [S2], and $a(\Bla)$ is a certain integer 
(see 5.1 for the notation).   
The main objective in this paper is the proof of their conjecture, 
based on the theory of character sheaves.  
Note that the conjecture (in the case where $\Bk = \BC$) was 
recently proved by Kato 
[Ka3], [Ka4] by a totally different method. 
\par
Our strategy for the proof is as follows; in [SS], 
we have constructed a complex $K_{T,\CE}$ on 
$\CX$, where 
$T$ is the $\th$-stable maximal torus of $G$ consisting of 
diagonal matrices, and $\CE$ is a tame local system on $T^{\io\th}$.
The construction in [SS] makes sense if we replace $T$ by 
an $F$-stable, $\th$-stable maximal torus contained in a $\th$-stable
(not necessarily $F$-stable) Borel subgroup of $G$.
By changing the notation, let $(T,\CE)$ be such a pair. 
Assume given an isomorphism $\vf_0: F^*\CE \isom \CE$. 
Then one can show that there exists a canonical isomorphism 
$\vf: F^*K_{T,\CE} \isom K_{T,\CE}$. 
We define an $H^F$-invariant function $\x_{T,\CE}$ on $\CX^F$ as 
the characteristic function $\x_{K_{T,\CE}, \vf}$. 
Then we define a Green function $Q_T : \CX\uni^F \to \Ql$ 
as the restriction of $\x_{T,\CE}$ on $\CX\uni^F$, which is shown 
to be independent of the choice of $\CE$. 
The Green functions have a similar role as 
in the theory of character sheaves  by Lusztig [L3]. 
We prove a character formula for $\x_{T,\CE}$, which describes the 
values of $\x_{T,\CE}$ in terms of various Green functions  associated 
to the centralizer of semisimple elements $s \in G^{\io\th}$ in $G$.
By using the character formula, we prove the orthogonality relations 
for $\x_{T,\CE}$, and also for Green functions $Q_T$.  
By the Springer correspondence established in [SS], the orthogonality 
relations for Green functions imply the orthogonality relations 
for the characteristic functions of the intersection cohomologies 
$\IC(\ol \CO_{\Bla},\Ql)$.   This fits to the characterization of 
modified Kostka polynomials ([S2]) in terms of certain orthogonality 
relations, and (*) follows from this, combined with a certain purity
result. 
\par
In the last section, we prove that the characteristic functions of
$F$-stable character sheaves give a basis of the space of 
$H^F$-invariant functions on $\CX^F$, which supports the conjecture 
in Introduction in [SS] that our character sheaves coincide with 
those proposed in [HT].   
\par\bigskip
{\bf Contents}
\par\medskip
1. \ Green functions 
\par
2. \ Character formulas 
\par
3. \ Orthogonality relations 
\par
4. \ A purity result 
\par
5. \ Kostka polynomials 
\par
6. \ $H^F$-invariant functions on $\CX^F$

\par\bigskip
\section{Green functions}

\para{1.1.}
Let 
$G = GL(V)$, $H = G^{\th} \simeq Sp_{2n}$ be as in 
Introduction.  Then 
$G, V$ have natural $\Fq$-structures, and denote by $F$ 
the corresponding Frobenius maps. The map $\th$  
commutes with $F$.  The $\th$-stable pair
$(B, T)$ in [SS, 1.2] is an $F$-stable
pair, which we denote in this paper by $(B_0, T_0)$ 
to distinguish their $\Fq$-structures. 
$N_G(T_0)/T_0$ is a Weyl group of $G$ 
which is isomorphic to $S_{2n}$.  
Let $W_n = N_H(T^{\th}_0)/T^{\th}_0$.
Then $W_n$ is a Weyl group of type $C_n$.
The image of $N_H(T_0) \subset N_G(T_0)$ under the map
$N_G(T_0) \to N_G(T_0)/T_0$ is isomorphic to 
$N_H(T_0)/T_0^{\th} \simeq N_H(T_0^{\th})/T_0^{\th}$, and 
the inclusion map  
$N_H(T_0) \to N_G(T_0)$ induces an embedding  
$W_n \hra S_{2n}$.  $F$ acts trivially on $W_n$
and on $S_{2n}$.
Let $\CW = N_H(T_0^{\io\th})/Z_H(T_0^{\io\th})$, which is isomorphic to $S_n$.
Since $W_n \simeq S_n \ltimes (\BZ/2\BZ)^n$, $w \in W_n$ can be written 
as $w = \s\t$ with $\s \in S_n$, $\t \in (\BZ/2\BZ)^n$.
We write $\t \in (\BZ/2\BZ)^n$ as $\t = (\t(1), \dots, \t(n))$ 
with $\t(i) = \pm 1$. 
Let $\{ e_i, f_i\}$ be a symplectic basis of $V$ as in [SS], consisting 
of eigenvectors for $T_0$.

\par
The $G^F$-conjugacy classes of $F$-stable maximal tori in $G$ are
parametrized by the conjugacy classes of $S_{2n}$.  For each 
$w \in S_{2n}$ and its representative $\dw \in N_G(T_0)$,  
choose $g \in G$ such that $g\iv F(g) = \dw$.
Then $T_w = gT_0g\iv$ gives a representative of the $G^F$-conjugacy class 
of $F$-stable maximal tori  
corresponding to the conjugacy class in $S_{2n}$ containing $w$.
If $T$ is an $F$-stable, $\th$-stable maximal torus contained 
in an $\th$-stable Borel subgroup of $G$, then $T$ is $H$-conjugate 
to $T_0$, and $T$ can be written as $T = T_w$ with $w \in W_n$.
In this case, $T^{\io\th}$ is $H$-conjugate to $T_0^{\io\th}$,  
and the $H^F$-conjugacy classes of $F$-stable, 
$\th$-anisotropic maximal tori in $G$ are 
parametrized by the conjugacy classes
of $N_H(T_0^{\io\th})/Z_H(T_0^{\io\th}) \simeq S_n$.
For $w = \s\t, w' = \s'\t' \in W_n$, $T_w^{\io\th}$ is $H^F$-conjugate 
to $T_{w'}^{\io\th}$ if and only if $\s, \s'$ are conjugate in 
$S_n$.  Note that $T_w^{\io\th}$ is isomorphic to an $F$-stable maximal
torus $S_{\s}$ in $GL_n$ (defined for $\s \in S_n$ similar to $T_w$), 
and the correspondence $T_w^{\io\th} \mapsto S_{\s}$ is compatible 
with the parametrization of $GL_n^F$-conjugacy classes of $F$-stable 
maximal tori in $GL_n$. 

\para{1.2.}
In the following we denote the objects $M_n, M_I$, etc. in [SS, 3.2]
by $M_{n,0}, M_{I,0}$, etc. by attaching 0. 
For a given $w \in W_n = N_H(T^{\th}_0)/T_0^{\th}$, take 
a representative $\dw \in N_H(T_0^{\th})$.  We choose an element 
$h \in H$ such that $h\iv F(h) = \dw$.  Put 
$T = hT_0h\iv, B = hB_0h\iv, M_n = h(M_{n,0})$.   Thus $M_n$ is 
a maximal isotropic subspace of $V$ stable by $B$. 
Let $I$ be a subset of $[1,n]$ such that $|I| = m$.
We put $M_I = h(M_{I,0})$, and consider the maps 
$\psi: \wt\CY \to \CY$, 
$\psi_I: \wt\CY_I \to \CY_m^0$ as in [SS, 3.2], 
defined in terms of $T, B, M_n, M_I$ instead of 
$T_0, B_0, M_{n,0}, M_{I,0}$.  
Let $\CE$ be a tame local system on $T^{\io\th}$.  As in [SS, 3.5], 
we consider the complex $(\psi_I)_*\a_I^*\CE$.  
We define a variety $\wt\CY^{\bullet}$ by 
\begin{equation*}
\wt\CY^{\bullet} = \{ (x,v, gT^{\th}) \in G^{\io\th}\reg \times V 
          \times H/T^{\th} \mid g\iv xg \in T^{\io\th}\reg, 
                  g\iv v \in M_n\}, 
\end{equation*}
and define a subvariety $\wt\CY_I^{\bullet}$ 
similarly, but  by replacing $M_n$ by $M_I$ 
in the above formula. 
We define maps $\psi^{\bullet}: \wt\CY^{\bullet} \to \CY,  
\psi_I^{\bullet}: \wt\CY_I^{\bullet} \to \CY_m^0$ by 
$(x,v, gT^{\th}) \mapsto (x,v)$.
Note that 
$Z_H(T^{\io\th}) \simeq \prod_{1 \le i \le n}SL_2$. 
We denote by $U_2$ the unipotent radical of a Borel subgroup of $SL_2$. 
Then $(x,v,gT^{\th}) \mapsto (x,v, g(B^{\th}\cap Z_H(T^{\io\th}))$
gives a map $\wt\g_I : \wt\CY^{\bullet}_I \to \wt\CY_I$, which gives rise to
a vector bundle over $\wt\CY_I$ with fibre isomorphic to 
$U_2^n$.   
Let $\a_0^{\bullet} : \wt\CY^{\bullet} \to T^{\io\th}$ be the map
defined by $(x,v, gT^{\th}) \mapsto g\iv xg$, and define 
$\a_I^{\bullet}: \wt\CY_I^{\bullet} \to T^{\io\th}$ by its 
restriction on $\wt\CY_I^{\bullet}$.
Then we have the following 
commutative diagram 
\begin{equation*}
\begin{CD}
T^{\io\th} @<\a_I^{\bullet}<<  \wt\CY_I^{\bullet} 
                          @>\psi^{\bullet}_I>> \CY_m^0 \\
 @V\id VV       @VV\wt\g_I V                    @VV\id V  \\
T^{\io\th} @<\a_I <<  \wt\CY_I    @>\psi_I>>          \CY_m^0.
\end{CD}
\end{equation*} 
It follows from this that we have a canonical isomorphism 
$(\psi_I^{\bullet})_!(\a_I^{\bullet})^*\CE[-2n] 
   \isom (\psi_I)_*\a_I^*\CE$
for each $I$ since $\dim U_2^n = n$, and we have 
\begin{equation*}
\tag{1.2.1}
\psi^{\bullet}_!(\a_0^{\bullet})^*\CE[-2n] \isom \psi_*\a_0^*\CE.
\end{equation*} 
\par
Let $\wt\CY^{\bullet}_{F(I)}$ be a similar variety as $\wt\CY_I^{\bullet}$,
defined in terms of $T, F(M_I)$ instead of $T, M_I$. 
The map 
$(x,v, gT^{\th}) \mapsto (F(x), F(v), F(g)T^{\th})$ gives 
a map $F: \wt\CY_I^{\bullet} \to \wt\CY_{F(I)}^{\bullet}$.  
(Here we use the notation $F(I)$ just as a symbol, and  
it does not mean a subset of $[1,n]$.  In fact,    
$F(M_I) = h(\dw(M_{I,0}))$ and $F(M_I) \not\subset M_n$ 
in general.)
Now $w$ is written as $w = \s\t \in W_n$  
with $\s \in S_n, \t \in (\BZ/2\BZ)^n$.
We define $\t_I \in (\BZ/2\BZ)^n$ for each $I$ by the condition   
\begin{equation*}
\t_I(i) = \begin{cases}
            -1 &\text{ if $i \in I$ and $\t(\s\iv(i)) = -1$ }, \\
             1 &\text{ otherwise}.
          \end{cases}
\end{equation*}
Let $\dot\t_I \in N_H(T_0^{\th})$  be a representative of $\t_I$ which is
a permutation of the basis $\{ e_i, f_i\}$ of $V$.  Then we have 
$\dot\t_I\dw(M_{I,0}) = M_{\s(I),0}$, where $\s(I)$ is a subset of $[1,n]$.
Put $h_I = h\dot\t_Ih\iv \in N_H(T^{\th})$.  
Then $h_I \in Z_H(T^{\io\th})$.
We have $h_IF(M_I) = M_{\s(I)}$.   
Hence one can define a map 
$b_I : \wt\CY_{F(I)}^{\bullet} \to \wt\CY_{\s(I)}^{\bullet}$
by $(x,v, gT^{\th}) \mapsto (x,v, gh_I\iv T^{\th})$, and 
we obtain a map $b_IF: \wt\CY_I^{\bullet} \to \wt\CY_{\s(I)}^{\bullet}$.
\par
Note that $\CY_m$ coincides with the original variety 
$\bigcup_{g \in H}g((T_0)\reg^{\io\th} \times M_{m,0})$, and 
$\CY_m^0 = \CY_m \backslash \CY_{m-1}$. 
Hence $\CY_m, \CY_m^0$ are  $F$-stable. 
Let us consider the variety 
$\wt\CY_m^{+,\bullet} = (\psi^{\bullet})\iv(\CY_m^0)$.  
Then as in the case of $\wt\CY_m^+$ (see [SS, 3.4]), 
we have $\wt\CY_m^{+,\bullet} = \coprod_I\wt\CY_I^{\bullet}$,
where $I$ runs over all the subsets of $[1,n]$ such that $|I| = m$, 
and the set $\wt\CY_I^{\bullet}$ forms connected components of 
$\wt\CY_m^{+,\bullet}$.
It follows from the above discussion, that one can define 
a map $F' = \coprod_Ib_IF$ on $\wt\CY_m^{+,\bullet}$.
We have a commutative diagram (note that $h_I \in Z_H(T^{\io\th})$),
where $\psi_m^{\bullet}$ is the restriction of $\psi^{\bullet}$ on
$\wt\CY_m^{+,\bullet}$.
\begin{equation*}
\tag{1.2.2}
\begin{CD}
T^{\io\th} @<\a_0^{\bullet}<<   \wt\CY_m^{+,\bullet}  
                 @>\psi_m^{\bullet}>> \CY_m^0 \\
@VF VV                 @VVF' V                   @VVF V    \\
T^{\io\th}  @<\a_0^{\bullet}<<   \wt\CY_m^{+,\bullet}  
              @>\psi_m^{\bullet} >> \CY_m^0.
\end{CD}
\end{equation*}
\par
Assume that $\CE$ is such that $F^*\CE \simeq \CE$, and   
fix an isomorphism 
$\vf_0: F^*\CE \isom \CE$. 
Put $K^{\bullet}_{m,T,\CE} = (\psi_m^{\bullet})_!(\a_0^{\bullet})^*\CE$
and $K_{m,T,\CE} = (\psi_m)_*\a_0^*\CE$.
Then by (1.2.2), $\vf_0$ induces an isomorphism
$F^*(K^{\bullet}_{m,T,\CE}) \isom K^{\bullet}_{m,T,\CE}$.
A similar formula as (1.2.1) holds by replacing $\psi^{\bullet}, \psi$
by $\psi^{\bullet}_m, \psi_m$, and it induces  
an isomorphism
$\vf'_m: F^*(K_{m,T,\CE})\isom K_{m,T,\CE}$. 

\para{1.3.}
By (3.5.2) (and the discussion in the proof of Proposition 3.6) 
in [SS], $(\psi_m)_*\a_0^*\CE$ is expressed as 
\begin{equation*}
(\psi_m)_*\a_0^*\CE \simeq \bigoplus_{0 \le i \le n-m}
              \CL_{m,i}[-2i],
\end{equation*}
where 
$\CL_m = \bigoplus_{\r \in \CA_{\Bm,\CE}\wg} \wt V_{\r}\otimes \CL_{\r}$
with a simple local system $\CL_{\r}$ on $\CY_m^0$, 
and $\CL_{m,i}$ is a direct sum of 
$\sharp\{ J \subset [1,n-m]\mid |J| = i\}$ copies of $\CL_m$. 
It follows that $\vf_m'$ induces an isomorphism 
$F^*\CL_{m,i} \isom \CL_{m,i}$ for each $i$, and in particular, 
an isomorphism $\vf''_m: F^*\CL_m \isom \CL_m$ by taking $i = n-m$. 
\par
Let $\CX_m$ be the variety defined in [SS, 3.1], which is the 
closure of $\CY_m$ in $\CX = G^{\io\th} \times V$.
We consider a semisimple perverse sheaf $K_{T,\CE}$ on $\CX$ 
defined as in (4.1.1) in [SS], 
\begin{equation*}
\tag{1.3.1}
K_{T,\CE} = \bigoplus_{0 \le m \le n}\bigoplus_{\r \in \CA_{\Bm, \CE}\wg}
           \wt V_{\r} \otimes \IC(\CX_m, \CL_{\r})[d_m],
\end{equation*}
where $d_m = \dim \CX_m$.
Thus $K_{T,\CE}$ can be written as 
\begin{equation*}
K_{T,\CE} = \bigoplus_{0 \le m \le n}\IC(\CX_m, \CL_m)[d_m], 
\end{equation*}
with a semisimple local system $\CL_m$ on $\CY_m^0$.
Here $\vf_m''$ is extended to a unique isomorphism 
$
\vf_m: F^*\IC(\CX_m,\CL_m)[d_m] \isom \IC(\CX_m,\CL_m)[d_m].
$ 
We now define an isomorphism $\vf: F^*K_{T,\CE} \isom K_{T,\CE}$
by $\vf = \bigoplus_{0 \le m \le n}\vf_m$.

\para{1.4.}
Following [SS, 4.1], we consider the diagram 
\begin{equation*}
\begin{CD}
T^{\io\th} @<\a <<  \wt\CX @>\pi >> \CX, 
\end{CD}
\end{equation*}
where $\CX = G^{\io\th} \times V$ and 
\begin{equation*}
\tag{1.4.1}
\wt\CX = \{ (x,v, gB^{\th}) \in G^{\io\th} \times V \times H/B^{\th}
                \mid g\iv xg \in B^{\io\th}, g\iv v \in M_n\},
\end{equation*}
with $\pi: (x,v, gB^{\th}) \mapsto (x,v)$, 
$\a: (x,v, gB^{\th}) \mapsto p(g\iv xg)$, ($p: B^{\io\th} \to T^{\io\th}$
is the natural projection).  
We also consider 
\begin{equation*}
\tag{1.4.2}
\wt\CX\uni = \{ (x,v, gB^{\th}) \in G^{\io\th}\uni \times 
 V \times H/B^{\th} \mid g\iv xg \in B^{\io\th}, g\iv v \in M_n\}
\end{equation*}
and a map $\pi_1: \wt\CX\uni \to \CX\uni = G^{\io\th}\uni \times V$
by $(x,v, gB^{\th}) \mapsto (x,v)$.
Hence $\wt\CX, \wt\CX\uni$ are the same objects as those given in 
[SS, 3.1, 2.4],  
defined by replacing $B_0, M_{n,0}$ by $B, M_n$.
Let $\wt K_{B, \CE} = \pi_*\a^*\CE[\dim \CX]$.
Note that $\wt K_{B,\CE}, K_{T,\CE}$ are respectively isomorphic 
to the corresponding objects given in [SS]. 
Hence by [SS, Theorem 4.2] we have $\wt K_{B,\CE} \simeq K_{T,\CE}$.
For each $0 \le m \le n$, let
$\CX_m = \bigcup_{g \in H}g(B^{\io\th} \times M_m)$ and 
put $\CX_m^0 = \CX_m \backslash \CX_{m-1}$.
Thus $\CX_m, \CX_m^0$ are the same objects as those defined 
in [SS, 4.3]. 
Put $\wt\CX_m^+ = \pi\iv(\CX_m^0)$.
The decomposition $\wt\CX_m^+ = \coprod_I\wt\CX_I$ in 
[SS, Lemma 4.4] is applicable to our setup, where $\wt\CX_I$ 
are defined in terms of $B, M_I$.  More generally, 
we consider a $\th$-stable Borel subgroup $B'$ containing 
$T$ and ${B'}^{\th}$-stable
maximal isotropic subspace $M_n'$ with respect to a symplectic basis 
$\{ e_i', f_i'\}$ consisting of eigenvectors for $T$.   Then 
the definition of $\wt\CX_I$ in [SS, 4.3] makes sense in this situation, 
and we denote the corresponding object by $\wt\CX_{I, B'}$, i.e.,
\begin{equation*}
\wt\CX_{I,B'} = \{ (x,v, g{B'}^{\th}) \in \wt\CX_m^+ 
                     \mid (x,v, g{B'}^{\th}) \mapsto I \},
\end{equation*}
where we understand that $\wt\CX_m^+$ is defined in terms of 
$B', M_n'$.
We fix $I \subset [1,n]$, and consider $\wt\CX_I = \wt\CX_{I,B}$.
The Frobenius action $(x,v, gB^{\th}) \mapsto 
(F(x), F(v), F(g)FB^{\th})$ gives a map 
$F: \wt\CX_{I,B} \to \wt\CX_{F(I), FB}$, 
where $\wt\CX_{F(I), FB}$ denotes the variety defined 
by the symplectic basis $\{ F(he_i), F(hf_i)\}$ in $V$.
Let $h_I \in N_H(T^{\th}) \cap Z_H(T^{\io\th})$ be as in 1.2.
One can define an isomorphism 
$b'_I : \wt\CX_{F(I), FB} \to \wt\CX_{\s(I), B'}$ 
by $(x,v,gFB^{\th}) \mapsto (x,v, gh_I\iv {B'}^{\th})$, 
where $B' = h_I(FB) h_I\iv$.   
Thus we have a map $b'_IF: \wt\CX_{I,B} \to \wt\CX_{\s(I), B'}$.
\par
We define maps  
$\pi_I: \wt\CX_{I,B} \to \CX_m^0, \wt\a_I : \wt\CX_{I, B} \to T^{\io\th}$
similar to the case $\wt\CX$.  
Since $h_I \in Z_H(T^{\io\th})$, we have the following commutative diagram
\begin{equation*}
\begin{CD}
T^{\io\th} @<\wt\a_I << \wt\CX_{I,B}  @>\pi_I >>  \CX_m^0 \\
  @V F VV          @Vb_I'F VV         @VV F V  \\
T^{\io\th}  @<\wt\a'_{\s(I)} <<  \wt\CX_{\s(I), B'}  @>\pi'_{\s(I)} 
   >>  \CX_m^0,
\end{CD}
\end{equation*}
where $\pi'_{\s(I)}, \wt\a'_{\s(I)}$ are defined similar to 
$\pi_I, \wt\a_I$, but by using $B'$ instead of $B$.
We consider the complex $(\pi_I)_*(\wt\a_I)^*\CE$ on $\CX_m^0$, 
which we denote by $\wt K_{I,B, \CE}$. 
Similarly we define 
$\wt K_{\s(I), B', \CE} = (\pi'_{\s(I)})_*(\wt\a'_{\s(I)})^*\CE$.  
By the above commutative diagram, $\vf_0: F^*\CE \isom \CE$ induces 
an isomorphism 
$\wt\vf_I : F^*\wt K_{\s(I), B',\CE} \isom \wt K_{I, B,\CE}$.
On the other hand,  we consider the complex 
$(\psi^{\bullet}_I)_*(\a^{\bullet}_I)^*\CE[-2n]$ 
on $\CY_m^0$, which is decomposed as in (3.4.3), (3.5.1) in [SS].
Thus there exists a unique DGM-extension to $\CX_m^0$, 
which we denote by $K^{\bullet}_{I,T, \CE}$.  
The discussion in 1.3 shows that 
$b_IF$ gives a map $\wt\CY^{\bullet}_I \to \wt\CY^{\bullet}_{\s(I)}$, 
and  $\vf_0$ induces an isomorphism 
$\vf_I : F^*K^{\bullet}_{\s(I),T, \CE} \to K^{\bullet}_{I,T, \CE}$. 
The proof of Theorem 4.2 in [SS], together with (1.2.1), 
 shows that there exists
an isomorphism 
$K^{\bullet}_{I,T, \CE}[-2n] \isom \wt K_{I, B,\CE}$ which we denote 
by $g_B$. The map $g_{B'}$ is defined similarly.
Then the following diagram commutes;
\begin{equation*}
\tag{1.4.1}
\begin{CD}
F^*K^{\bullet}_{\s(I),T, \CE}[-2n] @> \vf_I >> K^{\bullet}_{I,T, \CE}[-2n]  \\
@V F^*(g_{B'}) VV               @VV g_BV      \\
F^*\wt K_{\s(I), B',\CE}   @> \wt\vf_I >>  \wt K_{I,B,\CE} 
\end{CD}
\end{equation*}
 
\para{1.5.}
Let $\DD X$ be the derived category of $\Ql$-constructible sheaves
on a variety $X$ over $\Bk$. Assume that $X$ is defined over $\Fq$
with Frobenius map $F: X \to X$. 
Recall that for a given complex $K \in \DD X$ with an isomorphism 
$\f : F^*K \isom K$, 
the characteristic function $\x_{K,\f} : X^F \to \Ql$ is defined 
as
\begin{equation*}
\x_{K,\f}(x) = \sum_i(-1)^i\Tr(\f^*, \CH_x^iK) \qquad (x \in X^F), 
\end{equation*}
where $\f^*$ is the induced isomorphism on $\CH^i_xK$.
\par
Returning to our original setup, we consider a tame local
system $\CE$ on $T^{\io\th}$ such that $F^*\CE \simeq \CE$.
Since the isomorphism $F^*\CE \isom \CE$ is unique up to scalar, 
we fix $\vf_0: F^*\CE \isom \CE$ by the condition that 
the induced map of $\vf_0$ on the stalk $\CE_e$ at the unit element 
$e \in (T^{\io\th})^F$
is identity.   
We consider the characteristic function 
$\x_{K_{T,\CE},\vf}$ induced from this $\vf_0$, 
and denote it by $\x_{T,\CE}$.
Since $K_{T,\CE}$ is an $H$-equivariant perverse sheaf, 
$\x_{T,\CE}$ is an $H^F$-invariant function on 
$\CX^F = (G^{\io\th} \times V)^F$.
We define a function $Q_T = Q_T^G$ as the restriction of $\x_{T,\CE}$ on 
$\CX\uni^F = (G^{\io\th}\uni \times V)^F$, 
and call it a Green function on 
$\CX\uni^F$. 
This definition makes sense since the following proposition holds.

\begin{prop} 
The restriction of $\x_{T,\CE}$ on $\CX\uni^F$ is independent of 
the choice of $\CE$.  Hence it coincides with the restriction of 
$\x_{T,\Ql}$ on it.
\end{prop}

\begin{proof}
The corresponding fact in the case of reductive groups was 
first proved by Lusztig in [L3, (8.3.2)]. However its proof
contained a gap. The complete proof was later given in [L5]
in a more general setup for disconnected reductive groups. 
(See also [Le] for the Lie algebra analogue.)  In the discussion 
below, we modify the argument in [L5].  
Let $\pi_1: \wt\CX\uni \to \CX\uni$ be the map as in 1.4, and 
$\pi_{1, I, B}$ the restriction of $\pi_1$ on 
$\wt\CX_{I,B} \cap \wt\CX\uni$.
Since there is a canonical isomorphism 
$\wt K_{I,B, \CE}|_{\CX\uni} \simeq (\pi_{1,B,I})_*\Ql$
for any $\CE$, the diagram (1.4.1) is extended to the following 
commutative diagram.
\begin{equation*}
\begin{CD}
K^{\bullet}_{I,T,\CE}[-2n]|_{\CX\uni} @>g_B>>  \wt K_{I, B,\CE}|_{\CX\uni} 
                   @> \sim >> (\pi_{1,I,B})_*\Ql  \\
@A\vf_I AA     @AA\wt\vf_I A   @AA\Psi_I A    \\
F^*K^{\bullet}_{\s(I), T,\CE}[-2n]|_{\CX\uni}  @>F^*(g_{B'})>>   
   F^*\wt K_{\s(I), B',\CE}|_{\CX\uni}
                    @>\sim >>  F^*(\pi_{1,\s(I), B'})_*\Ql.   
\end{CD}
\end{equation*}
The construction of the map $\vf: F^*K_{T,\CE} \isom K_{T,\CE}$ in 
1.2, 1.3 shows 
that $\vf$ is determined completely from $\vf_I$ for each $I$.  
Hence in order to prove the proposition, it is enough to show that
\par\medskip\noindent
(1.6.1) \  The map $\Psi_I$ is independent of the choice of $\CE$.
\par\medskip

We show (1.6.1).  
In the discussion below, the objects 
$\wt\CX, \wt\CX_I$, etc. are
referred to $T, B$ unless otherwise stated.
Weyl groups appearing here such as $S_{2n}, W_n, \CW \simeq S_n$ are 
considered with respect to $T$.
Let $\wt\w : G \to T/S_{2n}$ be the Steinberg map
with respect to $G$.  We embed $S_n \simeq \CW$ in $S_{2n}$ as in 1.1.  Then
we have $T^{\io\th}/S_n \hookrightarrow T/S_{2n}$, and 
$\wt\w$ induces the Steinberg map $\w : G^{\io\th} \to T^{\io\th}/S_n$.
We have a commutative diagram
\begin{equation*}
\tag{1.6.2}
\begin{CD}
\wt \CX  @>\a >>  T^{\io\th}   \\
@V\pi VV          @VV\w V  \\  
\CX     @>\w_1>>  T^{\io\th}/S_n,
\end{CD}
\end{equation*}
where $\w_1$ is the composite of the projection 
$\CX \to G^{\io\th}$ with $\w$.
Replacing $\wt\CX, \CX$ by $\wt\CY,  \CY$, 
we have a similar 
diagram 
\begin{equation*}
\tag{1.6.3}
\begin{CD}
\wt\CY @>\a_0>>  T^{\io\th}\reg \\
@V\psi VV             @VV \w V     \\
\CY   @>\w_1>>  T^{\io\th}\reg/S_n.
\end{CD}
\end{equation*}
Let $\CZ = \CY \times_R T^{\io\th}\reg$ be the fibre product of $\CY$ 
and $T^{\io\th}$ over $R = T^{\io\th}\reg/S_n$, and 
$\vD: \wt\CY \to \CZ$ be the natural map.  Then $\psi$ is decomposed as
$\psi = p\circ \vD$, where $p$ is the projection $\CZ \to \CY$. 
Since $\CY = \coprod_{0 \le m \le n}\CY_m^0$,
we have $\CZ = \coprod_{0 \le m \le n}\CZ_m^0$, where 
$\CZ_m^0 = \CY_m^0 \times_RT^{\io\th}\reg$. 
Recall that $\wt\CY = \coprod_{0 \le m \le n}\wt\CY_m^+$. 
Then we have a similar commutative diagram as (1.6.3), by replacing
$\wt\CY, \CY$ by $\wt\CY_m^+, \CY_m^0$, and 
$\psi_m : \wt\CY_m^+ \to \CY_m^0$ is decomposed as 
$\psi_m = p_m\circ \vD_m$, where $\vD_m: \wt\CY_m^+ \to \CZ_m^0$ is 
the natural map and $p_m: \CZ_m^0 \to \CY_m^0$ is the projection.
Now $\wt\CY_m^+$ is decomposed as $\wt\CY_m^+ = \coprod_I\wt\CY_I$, 
and we have a map $\xi_I: \wt\CY_I \to \wh \CY_I$ (see [SS, 3.2]).
Then we have a commutative diagram 

\vspace{-1cm}
\begin{figure}[h]
\hspace*{1cm}
\setlength{\unitlength}{1mm}
\begin{picture}(80,80)(-6,40)
\put(-6,100) {$\wh\CY_I$}
\put(5, 101) {\vector (1,0) {30}}
\put(18,103){$_{\b_I}$}
\put(40, 100){$T\reg^{\io\th}$}
\put(-4, 96) {\vector (0,-1) {20}}
\put(-6, 70){$\CY_m^0$}
\put(42, 96) {\vector (0,-1) {20}}
\put(44, 86){$_{\w}$}
\put(5, 72){\vector (1,0) {30}}
\put(40, 70){$T^{\io\th}\reg/S_n$,}

\put(15, 85){$\CZ_m^0$}
\put(8, 94){$_{\wh\vD_I}$}
\put(26,94){$_{q_m}$}
\put(4, 80){$_{p_m}$}
\put(18, 74){$_{\w_1}$}
\put(-8, 86){$_{\e_I}$}
\put(0, 96){ \vector (3, -2) {12}}
\put(24, 88){\vector (3,2) {14}}
\put(14, 82){\vector (-3,-2) {12}}

\put(-48,86){(1.6.4)}
\end{picture}
\end{figure}

\vspace{-2.5cm}

\noindent
where $\wh\vD_I : \wh\CY_I \to \CZ^0_m$ is the natural map 
and $q_m$ is the projection.  Note that 
$p_m$ (resp. $\e_I$) is a finite Galois covering with group 
$S_n$ (resp. $S_I$).  
Since $\e_I$ is proper, $\wh\vD_I$ is also proper. 
Since $\dim \wh\CY_I = \dim \CY_m^0 = \dim \CZ_m^0$, 
$\wh\vD_I(\wh\CY_I) = \vD_I(\wt\CY_I)$ is an irreducible 
component of $\CZ_m^0$.  Since $\wh\vD_I(\wh\CY_I)$ are mutually 
disjoint, 
$\CZ_m^0 = \coprod_I\vD_I(\wt\CY_I)$ 
gives a decomposition of $\CZ_m^0$ into irreducible components.
The map $\wh\vD_I$ is $S_I$-equivariant, and the restriction of 
$p_m$ on $\wh\vD_I(\wh\CY_I)$ gives a finite Galois covering with 
group $S_I$.  It follows that $\wh\vD_I$ gives an isomorphism 
$\wh\CY_I \simeq \wh\vD(\wh\CY_I)$. 
\par
We consider the complex $(\vD_I)_*\a_I^*\CE$ on $\CZ^0_m$. 
Since  $(\xi_I)_*\a_I^*\CE \simeq H^{\bullet}(\BP_1^{I'})\otimes \b_I^*\CE$ 
(where $I'$ is the complement of $I$ in $[1,n]$, cf. 3.4, 3.5 in [SS]),  
we see that  
\begin{equation*}
\tag{1.6.5}
(\vD_I)_*\a_I^*\CE
         \simeq H^{\bullet}(\BP_1^{I'}) \otimes q_m^*\CE|_{\vD_I(\wt\CY_I)}. 
\end{equation*}
It follows that 
\begin{equation*}
\tag{1.6.6}
(\vD_m)_*\a_0^*\CE|_{\wt\CY_m^+} \simeq 
         H^{\bullet}(\BP_1^{n-m})\otimes q_m^*\CE.
\end{equation*}
Let $\CZ_m = \CY_m \times_R T^{\io\th}\reg$.  Then $\CZ_m$ is closed in 
$\CZ$, and $\CZ_m^0$ is an open dense smooth subset of 
$\CZ_m$.   We consider the intersection cohomology 
$\IC(\CZ_m, q_m^*\CE)$ on $\CZ_m$.
By using a similar argument as in the proof of Proposition 3.6 in [SS], 
(1.6.6) implies that 
\begin{equation*}
\tag{1.6.7}
\vD_*\a_0^*\CE[\dim \CZ] \simeq \bigoplus_{0 \le m \le n}
                \IC(\CZ_m, q_m^*\CE)[\dim \CZ_m]. 
\end{equation*}

\par
We now consider the fibre product $\CZ' = \CX \times_{R'} T^{\io\th}$, 
where $R' = T^{\io\th}/S_n$.  Since $\CX = \coprod_m \CX_m^0$, 
we have $\CZ' = \coprod_m (\CX_m^0 \times _{R'}T^{\io\th})$.
Then $\CZ_m' = \CX_m^0\times_{R'}T^{\io\th}$ contains $\CZ_m$ as an open
dense subset.   
We have a commutative diagram 
\begin{equation*}
\begin{CD}
\wt\CX_m^+ @>\a>>  T^{\io\th}  \\
@V \pi_m VV             @VV\w V      \\
\CX_m^0  @>>>  T^{\io\th}/S_n,
\end{CD}
\end{equation*}
where $\pi_m$ is the restriction of $\pi$ on $\wt\CX_m^+$.
Let $\vD_m': \wt\CX_m^+ \to \CZ'_m$ be the natural map, and 
$p_m': \CZ'_m \to \CX_m^0$ be the first projection so that 
$\pi_m = p_m'\circ \vD_m'$.
We show that
\begin{equation*}
\tag{1.6.8}
(\vD_m')_*\a^*\CE|_{\wt\CX_m^+} \simeq H^{\bullet}(\BP_1^{n-m})\otimes 
                \IC(\CZ'_m, q_m^*\CE).
\end{equation*}
Put $K = (\vD_m')_*\a^*\CE|_{\wt\CX^+_m}$.  Since $\vD_m'$ is proper, 
by the decomposition theorem $K$ is a direct sum of complexes 
of the form $A[i]$ for a simple perverse sheaf $A$ on 
$\CZ'_m$ with some shift $i$.
Since  $K|_{\CZ^0_m} \simeq (\vD_m)_*\a_0^*\CE|_{\wt\CY_m^+}$, 
$K|_{\CZ^0_m}$ is decomposed as in (1.6.6). 
Hence in order to prove (1.6.8), it is enough to show that 
$\supp A \cap \CZ^0_m \ne \emptyset$ for any direct summand $A[i]$
of $K$. Here $(\pi_m)_*\a^*\CE|_{\wt\CX_m^+} \simeq (p_m')_*K$, 
and $(p_m')_*K$ is written as a direct sum of $(p_m')_*A[i]$.  
Since $\w$ is a finite morphism, $p_m'$ is also finite.  
Again by the decomposition theorem, $(p_m')_*A[i]$ is a 
direct sum of the form $B[j]$ for a simple perverse sheaf 
$B$ on $\CX_m^0$.  Since $p_m'$ is finite, if 
$\dim \supp A < \dim \CZ'_m$, we must have 
$\dim \supp B < \dim \CX_m^0 = \dim \CZ'_m$ for any $B$ appearing 
in $(p_m')_*A[i]$. 
But Proposition 4.8 in [SS]  implies that any simple perverse sheaf 
(up to shift) appearing in the decomposition 
of $(\pi_m)_*\a^*\CE|_{\wt\CX_m^+}$ 
has its support $\CX_m^0$.   It follows that 
$\dim \supp A = \dim \CZ'_m$ for any $A$, and (1.6.8) holds. 

\par
Let $\vD_I': \wt\CX_I \to \CZ_m'$ be the restriction of 
$\vD_m'$ on $\wt\CX_I$.  
Then $\vD_I'$ is proper,  and $\CZ'_m = \bigcup_I\vD_I'(\wt\CX_I)$
gives a decomposition into irreducible components.  
By comparing (1.6.5) and (1.6.8), 
we have 
\begin{equation*}
\tag{1.6.9}
(\vD_I')_*\a^*\CE|_{\wt\CX_I} \simeq H^{\bullet}(\BP_1^{I'})
     \otimes \IC(\vD_I'(\wt\CX_I), q_m^*\CE|_{\vD_I(\wt\CY_I)}).
\end{equation*} 
Put $K_{I,\CE}^{\sharp} = (\vD_I')_*\a^*\CE|_{\wt\CX_I}$. 
Let $q_m': \CZ_m' \to T^{\io\th}$ be the second projection. 
Here we note the isomorphism
\begin{equation*}
\tag{1.6.10}
K^{\sharp}_{I,\CE} \simeq (q_m')^*\CE \otimes K^{\sharp}_{I,\Ql}.
\end{equation*}
In fact by the projection formula, 
\begin{equation*}
\begin{split}
(q'_m)^*\CE \otimes (\vD_I')_*\a^*\Ql
   &\simeq (\vD_I')_*((\vD_I')^*(q_m')^*\CE \otimes \a^*\Ql) \\
   &\simeq (\vD_I')_*(\a^*\CE \otimes \a^*\Ql)
   \simeq (\vD_I')_*\a^*\CE,
\end{split}
\end{equation*}
where $\a^*\CE, \a^*\Ql$ are regarded as local systems on 
$\wt\CX_I$.
Hence (1.6.10) holds. 
\par
The construction of $K^{\sharp}_{I,\CE}$ makes sense for 
any $\th$-stable Borel subgroup $B'$ containing $T$.  We denote 
this complex by $K^{\sharp}_{I,B',\CE}$ to denote the dependence for $B'$.
Since $\pi_I = p'_m\circ \vD'_I$, we have 
$\wt K_{I,B,\CE} \simeq (p'_m)_*K^{\sharp}_{I, B,\CE}$, 
and similarly, 
$\wt K_{\s(I),B',\CE} \simeq (p_m')_*K^{\sharp}_{\s(I),B',\CE}$. 
In a similar way as the construction of 
$\wt\vf_I: F^*\wt K_{\s(I),B',\CE} \isom \wt K_{I,B,\CE}$, one can 
define an isomorphism  
$\F_I : F^*K^{\sharp}_{\s(I),B',\CE} \isom K^{\sharp}_{I, B,\CE}$
such that $(p'_m)_*(\F_I) = \wt\vf_I$ by using (1.6.5) and (1.6.9).
By (1.6.10), $\F_I$ can be expressed as 
\begin{equation*}
\F_I = \F_{\CE} \otimes \F_0 : F^*(q'_m)^*\CE \otimes 
       F^*K^{\sharp}_{\s(I),B',\Ql}  \isom (q'_m)^*\CE\otimes 
                   K^{\sharp}_{I, B,\Ql},
\end{equation*}
where $\F_{\CE} : F^*(q'_m)^*\CE \isom (q'_m)^*\CE$ and 
$\F_0: F^*K^{\sharp}_{\s(I),B',\Ql} \isom K^{\sharp}_{I,B,\Ql}$ are the  
natural maps induced from $F^*\CE \isom \CE$ and $F^*\Ql \isom \Ql$. 
In particular, $\F_0$ is independent of the choice of $\CE$. 
In order to prove (1.6.1), it is enough to show that $\F_I|_{\CZ'\uni}$
is independent of the choice of $\CE$.  (Here  
$\CZ\uni' = \CX\uni \times_{R'} T^{\io\th}$, and $\F|_{\CZ'\uni}$ 
denotes, for a map $\F$ on complexes on $\CZ'_m$, 
the map on the complexes restricted on $\CZ_m' \cap \CZ\uni'$ 
induced from $\F$.)
We have $\F_I|_{\CZ'\uni} \simeq \F_{\CE}|_{\CZ'\uni}\otimes
          \F_0|_{\CZ'\uni}$ and $\F_0|_{\CZ'\uni}$ is independent 
of $\CE$. 
Here $\CZ'\uni \simeq \CX\uni \times_{R'}\{ e\} \simeq \CX\uni$, 
where $e$ is the unit
element in $T^{\io\th}$.  
By definition, $\vf_0: F^*\CE \isom \CE$ gives an identity map on 
the stalk $\CE_e\simeq \Ql$.  It follows that the map 
$\F_{\CE}|_{\CZ'\uni}$ gives an identity map on $\CZ'\uni$, hence 
is independent of $\CE$.  This proves (1.6.1) and so the proposition 
follows.   
\end{proof}

\para{1.7.}
Take $w \in \wt\CW_{\CE} = \CW_{\CE}\ltimes (\BZ/2\BZ)^{n}$.
Under the embedding $\wt\CW_{\CE} \subset \wt\CW \simeq W_n$, 
we regard $w$ as an element of 
$W_n = N_H(T_0^{\th})/T_0^{\th} = N_H(T_0)/T_0^{\th}$, 
and let $T = T_w = hT_0h\iv$ for 
$h \in H$ such that $h\iv F(h) = \dot w \in N_H(T_0)$.  
We consider the complex $K_{T,\CE}$ and the map
$\vf: F^*K_{T,\CE} \isom K_{T,\CE}$.  
Let $\CE_0$ be the tame local system on $T_0^{\io\th}$
defined by $\CE_0 = (\ad h)^*\CE$.  
For later use, we shall describe the map $\vf$ by making 
use of $K_{T_0, \CE_0}$.
We follow the notation in 1.2. 
We write the varieties $\wt\CY^{\bullet}, \wt\CY_I^{\bullet}$, etc. 
in 1.2, as 
$\wt\CY_{T}^{\bullet},  \wt\CY_{I,T}^{\bullet}$, etc. 
 to indicate their dependence on  $T$.
We consider the variety $\wt\CY_{I,T_0}^{\bullet}$, which 
is defined by using $T_0$ and $M_{I,0}$.  Thus 
$\wt\CY_{I,T_0}^{\bullet}$ has a natural Frobenius 
action $F : (x,v, gT^{\th}_0) \mapsto (F(g), F(v), F(g)T^{\th}_0)$.
The map $(x,v,gT^{\th}) \mapsto (x,v, ghT_0^{\th})$ gives
an isomorphism $\d_I : \wt\CY_{I,T}^{\bullet} \to \wt\CY^{\bullet}_{I,T_0}$
commuting with the projections to $\CY_m^0$.
We define a map 
$a_I: \wt\CY_{I, T_0}^{\bullet} \to \wt\CY_{\s(I),T_0}^{\bullet}$
by $(x,v,gT^{\th}_0) \mapsto (x,v, g(\dot\t_I\dw)\iv T_0^{\th})$. 
Note that $a_I$ is well-defined since $\dot\t_I\dw(M_{I,0}) = M_{\s(I),0}$.
Then we have a commutative diagram
\begin{equation*}
\tag{1.7.1}
\begin{CD}
\wt\CY_{I,T}^{\bullet}  @>\d_I>> \wt\CY^{\bullet}_{I, T_0} \\
   @Vb_IF VV                           @VVa_IF V    \\
\wt\CY_{\s(I),T}^{\bullet}  @>\d_{\s(I)}>>  \wt\CY_{\s(I), T_0}^{\bullet}.
\end{CD}
\end{equation*}
Recall the complex $K^{\bullet}_{I,T,\CE}$ in 1.4, 
and define $K^{\bullet}_{I, T_0,\CE_0}$ similarly.
Since $\d_I^*K^{\bullet}_{I,T_0,\CE_0}$ is canonically isomorphic to 
$K^{\bullet}_{I,T,\CE}$, 
the isomorphism $\d_I$ induces an isomorphism 
$\d_I': K^{\bullet}_{I, T,\CE} \isom K^{\bullet}_{I, T_0,\CE_0}$.
By (1.7.1), we have a commutative diagram
\begin{equation*}
\tag{1.7.2}
\begin{CD}
F^*K^{\bullet}_{\s(I), T,\CE} @>F^*\d'_{\s(I)} >> 
                          F^*K^{\bullet}_{\s(I), T_0,\CE_0} \\
@V\vf_IVV                             @VV\vf'_I\circ F^*(\th_I) V  \\
K^{\bullet}_{I, T, \CE}  @>\d'_I>>  K^{\bullet}_{I, T_0, \CE_0},
\end{CD}
\end{equation*}
where $\vf_I': F^*K^{\bullet}_{I,T_0,\CE_0} \to K^{\bullet}_{I,T_0,\CE_0}$ 
is the natural map
induced from the Frobenius map on $\wt\CY^{\bullet}_{I,T_0}$, 
and $\th_I: K^{\bullet}_{\s(I), T_0,\CE_0} \to K^{\bullet}_{I,T_0,\CE_0}$ 
is the canonical isomorphism induced from $a_I$. 
Since $K^{\bullet}_{m, T,\CE} = \bigoplus_IK^{\bullet}_{I, T,\CE}$ and 
similarly for $K^{\bullet}_{m,T_0,\CE_0}$,  
we have an isomorphism 
$\d'_m: K^{\bullet}_{m,T,\CE} \isom K^{\bullet}_{m, T_0,\CE_0}$
induced from various $\d_I'$.  We have a commutative diagram as in (1.7.2)
with respect to $F^*K^{\bullet}_{m,T,\CE}, F^*K^{\bullet}_{m,T_0,\CE_0}$, 
$K^{\bullet}_{m,T\CE}, K^{\bullet}_{m,T_0,\CE_0}$.  
Then the map 
$\vf_m = \sum_I\vf_I  :F^*K^{\bullet}_{m, T,\CE} \isom K^{\bullet}_{m, T,\CE}$ 
is transferred 
to the map $\vf'_m\circ F^*(\th_{w\iv}):  F^*K^{\bullet}_{m, T_0,\CE_0} 
              \isom K^{\bullet}_{m, T_0,\CE_0}$, 
through $\d_m'$ and $F^*\d_m'$, 
where $\vf_m'$ is the isomorphism induced from the Frobenius map on 
$\wt\CY_{m, T_0}^{+,\bullet}$, and $\th_{w\iv} = \sum_I\th_I$. 
\par
Recall that $\CA_{\CE_I} = \End ((\e_I)_*\CE_I)$ and its extended 
algebra $\wt\CA_{\CE_I}$ in [SS, 3.5] for $\CE = \CE_0$. 
Note that $\CA_{\CE_I} \simeq \Ql[\CW_{\CE_I}]$ and 
$\wt\CA_{\CE_I} \simeq \Ql[\wt\CW_{\CE_I}]$. 
In [SS, 3.5], we have defined an algebra homomorphism
$\wt\CA_{\CE_I} \to \End((\psi_I)_*\a_I^*\CE_0)$.  Since 
$(\psi_I)_*\a_I^*\CE \simeq K_{I,T,\CE_0}^{\bullet}$ up to shift, 
we have an algebra homomorphism 
$\wt\CA_{\CE_I} \to \End(K^{\bullet}_{I,T_0,\CE_0})$.  
This action is described as follows; 
for each $w = \s\t \in \wt\CW_I \simeq S_I \ltimes (\BZ/2\BZ)^n$
($\s \in S_I, \t \in (\BZ/2\BZ)^n$),
put $w' = \s\t'$, where $\t'$ is the projection of $\t$ on 
$(\BZ/2\BZ)^{I'}$.  Then the map 
$a_w: (x,v, gT_0^{\th}) \mapsto (x,v, g w'T_0^{\th})$ gives an 
isomorphism on $\wt\CY_I^{\bullet}$, and $w \mapsto a_w$ gives 
a homomorphism $\wt\CW_I \to \Aut (\wt\CY_I^{\bullet})$. 
If $w \in \wt\CW_{\CE_I}$, $a_w$
induces an isomorphism on $K_{I, T_0, \CE}^{\bullet}$, and this gives 
the representation $\wt\CA_{\CE_I} \to \End(K^{\bullet}_{I, T_0, \CE})$.
The representation $\wt\CA_{\CE} \to \End (K_{m, T_0,\CE}^{\bullet})$
is obtained by inducing up the representation 
of $\wt\CA_{\CE_I}$ on $K^{\bullet}_{I,T_0, \CE_0}$ 
for $I = [1,m]$ to $\wt\CA_{\CE}$.
For each $w \in \wt\CW_{\CE}$ we denote by 
$\th_w \in \Aut(K^{\bullet}_{m,T_0,\CE_0})$ the image of $w$ under this map.
Then it is clear that $\th_{w\iv}$ in this context is exactly the same as 
$\th_{w\iv} = \sum_I \th_I$ defined before. 

\par
The isomorphism $\d'_m$ induces an isomorphism 
$K_{m,T,\CE} \isom K_{m,T_0,\CE_0}$. By a similar argument as in 1.3, 
one can define an isomorphism $\d': K_{T,\CE} \isom K_{T_0,\CE_0}$
extending $\d_m'$.
Let $\vf_T$ be the isomorphism $\vf : F^*K_{T,\CE} \isom K_{T,\CE}$, 
and $\vf_{T_0}$ the corresponding isomorphism for $K_{T_0,\CE_0}$. 
The automorphism $\th_w$ on $K^{\bullet}_{m,T_0,\CE_0}$ induces 
a unique automorphism on $K_{T_0,\CE_0}$ which we denote by the 
same symbol. 
Again by a similar argument as in 1.3, we have the 
following commutative diagram (note that $T = T_w$).
\begin{equation*}
\tag{1.7.3}
\begin{CD}
F^*K_{T,\CE} @>F^*\d'>>  F^*K_{T_0,\CE_0} \\
@V\vf_T VV            @VV \vf_{T_0}\circ F^*(\th_{w\iv}) V      \\
K_{T,\CE}   @>\d'>>      K_{T_0,\CE_0}.  
\end{CD}
\end{equation*}

\remark{1.8.}
Under the setup in Section 1 of [SS], we consider 
the complex $\pi_*\a^*\CE[\dim G^{\io\th}]$ on $G^{\io\th}$ 
associated to a tame local system $\CE$ on $T^{\io\th}$, which is
isomorphic to the complex $K_{T,\CE}$ by Theorem 1.16 in [SS].
We denote this $K_{T,\CE}$ as $K^{\sym}_{T,\CE}$ to distinguish 
$K_{T,\CE}$ on $G^{\io\th} \times V$.   
By definition, we have 
\begin{equation*}
\tag{1.8.1}
K^{\sym}_{T,\CE}[2n] \simeq 
   K_{T,\CE}|_{G^{\io\th} \times \{0\}}.
\end{equation*} 
As in the case of $K_{T,\CE}$, one can define $K^{\sym}_{T,\CE}$
for an $F$-stable maximal torus $T$ containing $\th$-stable maximal 
anisotropic torus, and an $F$-stable tame local system $\CE$ on 
$T^{\io\th}$. Then for each $\vf_0: F^*\CE \isom \CE$, one can construct 
an isomorphism $\vf : F^*K^{\sym}_{T,\CE} \isom K^{\sym}_{T,\CE}$, 
and we define a function $\x^{\sym}_{T,\CE}$ on $(G^{\io\th})^F$.
Clearly, $\vf$ is the restriction of the corresponding map for 
$K_{T,\CE}$ under (1.8.1), hence $\x_{T,\CE}^{\sym}$ coincides with 
$\x_{T,\CE}|_{(G^{\io\th} \times \{0\})^F}$. 
In particular, the statement in Proposition 1.6 holds also
for $\x^{\sym}_{T,\CE}$. 
We define a Green function 
$Q_T^{\sym}$ as the restriction of $\x^{\sym}_{T,\CE}$ on 
$(G^{\io\th}\uni)^F$, which is an $H^F$-invariant function on 
$(G^{\io\th}\uni)^F$.

\par\bigskip
\section{Character formulas}   

\para{2.1.}
We follow the notation in 1.2. 
Let $\ol\psi_m : \psi\iv(\CY_m) \to \CY_m$ be the restriction of
$\psi : \wt\CY \to \CY$ on $\psi\iv(\CY_m)$.  $\CY_{m-1}$ 
is a closed subset of $\CY_m$ and let 
$j': \CY_m^0 = \CY_m \backslash \CY_{m-1} \to \CY_m$
be the inclusion map.   Let $\wt\CE_m$ (resp. $\wt\CE'_m$) 
be the restriction of  
$\a_0^*\CE$ on $\psi\iv(\CY_m)$ (resp. on $\wt\CY_m^+$).
Then 
$(j'_!(\psi_m)_*\wt\CE_m', (\ol\psi_m)_*\wt\CE_m, 
   (\ol\psi_{m-1})_*\wt\CE_{m-1})$
is a canonical distinguished triangle in $\DD\CY_m$.
By [SS, (3.6.1), (3.6.3)], 
$(\ol\psi_m)_*\wt\CE_m$, $(\ol\psi_{m-1})_*\wt\CE_{m-1}$
are semisimple complexes, whose simple components are
constructible sheaves with even degree shifts.  
It follows that ${}^pH^i((\ol\psi_m)_*\wt\CE_m) = 0$ for odd $i$, 
where ${}^pH^iK$ is the $i$-th perverse cohomology of $K \in \DD \CY_m$.
The discussion in the proof of Proposition 3.6 in [SS] shows that 
the map 
$R^{2i}(\ol\psi_m)_*\wt\CE_m \to R^{2i}(\ol\psi_{m-1})_*1\wt\CE_{m-1}$
is surjective.   This implies that 
${}^pH^i(j'_!(\psi_m)_*\wt\CE_m') = 0$ for odd $i$. 
Hence the perverse cohomology long exact sequence associated
to the above distinguished triangle gives rise to a short exact 
sequence for each even degree part.   
It follows that ${}^pH^{2i}(j'_!(\psi_m)_*\wt\CE_m')$ is 
a semisimple complex given as 
\begin{equation*}
\tag{2.1.1}
{}^pH^{2i}(j'_!(\psi_m)_*\wt\CE_m') 
    \simeq \bigoplus_{\r \in \CA_{\Bm,\CE}\wg}
         \Ind_{\wt\CA_{\Bm,\CE}}^{\wt\CA_{\CE}}
    (H^{2i}(\BP_1^{n-m})\otimes\r)\otimes\IC(\CY_m,\CL_{\r}),
\end{equation*}
where $\IC(\CY_m,\CL_{\r})$ is a constructible
sheaf. 
Let $\psi^{\bullet} : \wt\CY^{\bullet} \to \CY$ and 
$\psi^{\bullet}_m : \wt\CY^{+\bullet}_m \to \CY_m^0$ 
 be the maps defined in 2.2, and $\ol\psi^{\bullet}_m$ 
be the restriction of $\psi^{\bullet}$
on $(\psi^{\bullet})\iv(\CY_m)$.  By abbreviation we denote 
the restriction of $\a_0^*\CE$ on 
$(\psi^{\bullet})\iv(\CY_m)$ and $\wt\CY_m^{+,\bullet}$ by 
the same symbols $\wt\CE_m, \wt\CE_m'$. 
Then $(j'_!(\psi_m^{\bullet})_*{\wt\CE'_m}, 
           \ol\psi_m^{\bullet}\wt\CE_m, 
               \ol\psi_{m-1}^{\bullet}\wt\CE_{m-1})$
is also a distinguished triangle.
Note that 
$j'_!(\psi_m^{\bullet})_*\wt\CE'_m[-2n]\simeq j'_!(\psi_m)_*\wt\CE_m'$
and $(\ol\psi^{\bullet}_m)_*\wt\CE_m[-2n] \simeq 
          (\ol\psi_m)_*\wt\CE_m$.   
\par
Let $\ol\pi_m$ be the restriction of $\pi: \wt\CX \to \CX$ 
on $\pi\iv(\CX_m)$.  $\CX_{m-1}$ is a closed subset of $\CX_m$,
and let $j$ be the inclusion map from 
$\CX_m^0 = \CX_m \backslash \CX_{m-1}$ to $\CX_m$.
We again denote by $\wt\CE_m$ (resp. $\wt\CE_m'$) the 
restriction of $\a^*\CE$ on $\pi\iv(\CX_m)$ (resp. on $\wt\CX_m^+$).
Then 
$(j_!(\pi_m)_*\wt\CE'_m, (\ol\pi_m)_*\wt\CE_m, 
        (\ol\pi_{m-1})_*\wt\CE_{m-1})$
gives rise to a canonical distinguished triangle on $\CX_m$.
By [SS, (4.9.1)], $(\ol\pi_m)_*\wt\CE_m$ is a semisimple 
complex, whose components are of the form $A[2i]$ for a perverse sheaf 
$A$. 
It follows that ${}^pH^i((\ol \pi_m)_*\wt\CE_m) = 0$ for odd $i$.
We consider the long exact sequence of perverse cohomologies arising 
from the distinguished triangle 
$(j_!(\pi_m)_*\wt\CE'_m, (\ol\pi_m)_*\wt\CE_m, 
        (\ol\pi_{m-1})_*\wt\CE_{m-1})$.
The map ${}^pH^{2i}(j_!(\pi_m)_*\wt\CE_m') \to 
    {}^pH^{2i}((\ol\pi_m)_*\wt\CE_m)$ is injective, 
and 
the map ${}^pH^{2i}((\ol\pi_{m-1})_*\wt\CE_{m-1}) \to 
        {}^pH^{2i+1}(j_!(\pi_m)_*\wt\CE_m')$ is surjective.
Since the restriction of ${}^pH^{2i+1}(j_!(\pi_m)_*\wt\CE_m')$
on $\CY_m$ is zero by the previous discussion, we see that 
${}^pH^{2i+1}(j_!(\pi_m)_*\wt\CE_m') = 0$.
It follows that the long exact sequence turns out to be 
a short exact sequence for ${}^pH^{2i}$-factors.
Thus ${}^pH^{2i}(j_!(\pi_m)_*\wt\CE_m')$ is a semisimple perverse 
sheaf, and by (4.9.1) in [SS], we have
\begin{equation*}
\tag{2.1.2}
{}^pH^{2i}(j_!(\pi_m)_*\wt\CE_m') \simeq 
\bigoplus_{\r \in \CA_{\Bm,\CE}\wg}
         \Ind_{\wt\CA_{\Bm,\CE}}^{\wt\CA_{\CE}}
    (H^{2i}(\BP_1^{n-m})\otimes\r)\otimes\IC(\CX_m,\CL_{\r}),  
\end{equation*}
and ${}^pH^{\odd}(j_!(\pi_m)_*\wt\CE_m') = 0$. 
\par
Put $d_m = \dim \CX_m$, and 
$K_m = j_!(\pi_m)_*\wt\CE_m'[d_m]$, 
$\ol K_m = (\ol\pi_m)_*\wt\CE_m[d_m]$.
Let $\vf: F^*K_{T,\CE} \isom K_{T,\CE}$ be the isomorphism 
defined in 1.3.
Then it follows from the discussion there, $\vf$ induces 
an isomorphism 
$\ol\vf_m: F^*\ol K_m \isom \ol K_m$.
It also induces an isomorphism 
$F^*((\pi_m)_*\wt\CE_m') \isom (\pi_m)_*\wt\CE_m'$.
Since $j$ commutes with $F$, we have an isomorphism 
$F^*K_m \isom K_m$ which is denoted by $\vf_m$. 
Then the above distinguished triangle is compatible with 
the isomorphisms $(\vf_m, \ol\vf_m, \ol\vf_{m-1})$.
By the property of distinguished triangles, we have 
\begin{equation*}
\tag{2.1.3}
(-1)^{d_m}\x_{\ol K_m,\ol\vf_m} = 
(-1)^{d_m}\x_{K_m,\vf_m} + (-1)^{d_{m-1}}\x_{\ol K_{m-1}, \ol\vf_{m-1}}.
\end{equation*}
\par
On the other hand, we consider a natural spectral sequence
\begin{equation*}
E_2^{i,j} = \CH^i({}^pH^jK_m) \Rightarrow \CH^{i+j}K_m
\end{equation*}
in the category of mixed constructible sheaves on $\CX_m$.  Taking 
a stalk $z \in \CX_m$, we get a spectral sequence 
$\CH^i_z({}^pH^jK) \Rightarrow \CH^{i+j}_zK_m$.
The isomorphism $\vf_m: F^*K_m \isom K_m$ induces an isomorphism 
$F^*({}^pH^jK_m) \isom {}^pH^jK_m$ which we denote also by $\vf_m$.
By the above spectral sequence, this implies that 
\begin{equation*}
\x_{K_m,\vf_m} = \sum_{j \ge 0}\x_{{}^pH^jK_m, \vf_m}
\end{equation*} 
Here we note, by [SS, Lemma 3.3], that 
\begin{equation*}
\tag{2.1.4}
d_m = \dim G^{\io\th} + 2m.
\end{equation*}
Hence $d_m-d_{m-1} = 2$, and the characteristic function $\x_{K,\vf}$ for 
$K = K_{T,\CE} = \ol K_n$ can be computed as
\begin{equation*}
\tag{2.1.5}
\x_{K,\vf} = \x_{T,\CE} = \sum_{m = 0}^n \sum_{j \ge 0}\x_{{}^pH^jK_m, \vf_m}.
\end{equation*} 

\para{2.2.}
For a semisimple element $s \in (G^{\io\th})^F$,
we consider $Z_H(s) \times V$.
We follow the notation in [SS, 4.3].  Thus 
$V = V_1\oplus\cdots\oplus V_t$, where $V_i$ is an eigenspace of 
$s$ with $\dim V_i = 2n_i$, and $F$ permutes the eigenspaces. 
$Z_G(s) \simeq G_1 \times \cdots \times G_t$, where 
$G_i \simeq GL(V_i)$ is a $\th$-stable subgroup of $G$
such that $G_i^{\th} \simeq Sp(V_i)$. Hence 
\begin{equation*}
\tag{2.2.1}
(Z_G(s))^{\io\th} \times V \simeq 
        \prod_{i=1}^t (G^{\io\th}_i \times V_i), 
\end{equation*}
and the natural action of $Z_H(s)$ on the left hand side 
is compatible with the natural action of $G_i^{\th}$ on 
$G_i^{\io\th} \times V_i$ of the right hand side under the
isomorphism $Z_H(s) \simeq G_1^{\th} \times\cdots \times G_t^{\th}$.  
\par
Let $T$ be a $\th$-stable, $F$-stable maximal torus of $G$ contained 
in a $\th$-stable (not necessarily $F$-stable) Borel subgroup $B$ of $G$. 
As in [SS, 1.17], we define
\begin{equation*}
\CM_s = \{ g \in H \mid g\iv sg \in B^{\io\th}\},
\end{equation*}
on which $Z_H(s) \times B^{\th}$ 
acts naturally.  The set $Z_H(s)\backslash \CM_s/B^{\th}$ is 
identified with $Z_H(s)\backslash\CM'_s/T^{\th}$ since it is 
labelled by $\vG = W_{H,s'}\backslash W_H$ (see [SS, 1.17]), 
where $\CM_s' = \{ g \in H \mid g\iv sg \in T^{\io\th}\}$.  
Hence $F$ acts naturally  
on this set.  Assume the orbit $\CO_{\g}$ corresponding to 
$\g \in \vG$ is $F$-stable, and we choose a representative
$x_{\g} \in \CO_{\g}^F \subset H^F$.   
Put $B_{\g} = Z_G(s) \cap x_{\g}Bx_{\g}\iv$, and 
$T_{\g} = x_{\g}Tx_{\g}\iv$.  Then $B_{\g}$ is a $\th$-stable 
Borel subgroup of $Z_G(s)$ containing $T_{\g}$ which is a
$\th$-stable maximal torus of $Z_G(s)$. 
Let $M_n$ be a $B^{\th}$-stable maximal isotropic subspace 
of $V$, and put 
$M_n^{\g} = x_{\g}(M_n)$.  Then $M_n^{\g}$ is $s$-stable, 
and is decomposed into eigenspaces of $s$, 
$M_n^{\g} = \bigoplus_{i = 1}^tM_{n_i}^{\g}$, where 
$M_{n_i}^{\g} = M_n^{\g} \cap V_i$ is a maximal isotropic subspace
of $V_i$ stable by $B_{\g}^i$.  (Here we denote by $B_{\g}^i$ the 
$i$-th factor of $B_{\g}$ under the identification 
$Z_G(s) \simeq G_1 \times \cdots \times G_t$). 
We define

\begin{equation*}
\wt\CX'_{\g} = \{ (x,v, gB^{\th}_{\g}) \in Z_G(s)^{\io\th} \times V 
                  \times Z_H(s)/B_{\g}^{\th} \mid 
             g\iv xg \in B^{\io\th}_{\g}, g\iv v \in M_n^{\g}\}, 
\end{equation*}
and a map $\pi_{\g}: \wt\CX'_{\g} \to \CX' = Z_G(s)^{\io\th} \times V$
by $(x,v, gB_{\g}^{\th}) \to (x,v)$.
In view of the decomposition (2.2.1) and the decomposition 
$Z_H(s) \simeq G_1^{\th} \times \cdots \times G^{\th}_t$, the variety
$\wt\CX'_{\g}$ is isomorphic to the product of the varieties 
$\wt\CX'_{\g, i}$, 
which is a similar variety as $\wt\CX$ in 1.2, defined   
with respect to $G^{\io\th}_i \times V_i$, and the maps $\pi_{\g}$
is compatible with the corresponding map for each $\wt\CX'_{\g,i}$.  
Thus one can define a complex $K_{T_{\g}, \CE_{\g}}^{Z_G(s)}$ on 
$\CX'$ similar to $K_{T,\CE}$ on $\CX$ ($\CE_{\g}$ is a tame local system 
on $T_{\g}^{\io\th}$).  The results of Section 1 can be applied for this 
general setting, by replacing $G$ by $Z_G(s)$. 
In particular, since $T_{\g}$ is an $F$-stable maximal torus of $Z_G(s)$, 
one can define a Green function $Q_{T_{\g}}^{Z_G(s)}$ on 
$(Z_G(s)^{\io\th}\uni \times V)^F$.  

\para{2.3.}
For an $F$-stable torus $S$, put $(S^F)\wg = \Hom (S^F, \Ql)$.
We consider the map $\e : S \to S$, by $t \mapsto t\iv F(t)$, 
which gives the Lang covering $S \to S/S^F$. Then $\e_*\Ql$ is 
decomposed as $\e_*\Ql \simeq \bigoplus_{\vth \in (S^F)\wg}\CE_{\vth}$. 
Here $\CE = \CE_{\vth}$ is an $F$-stable tame local system on $S$, and 
it is characterized by the property that the characteristic function 
$\x_{\CE, \vf_0}$ coincides with $\vth$ on $S^F$ up to scalar.
Thus if we choose $\vf_0: F^*\CE \isom \CE$ so that it induces an 
identity on $\CE_e$ (the stalk at the identity element $e \in S^F$), 
we have $\x_{\CE, \vf_0} = \vth$.   
\par
The following result is an analogy of Lusztig's character formula
[L3, Theorem 8.5].
(In the following, we use the notation 
$T^{\io\th, F} = T^{\io\th} \cap T^F$).   

\begin{thm}[Character formula]  
Let $s,u \in (G^{\io\th})^F$ be such that $su = us$, with 
$s$: semisimple and  $u$: unipotent.  
Assume that $\CE$ is an $F$-stable tame local system on $T^{\io\th}$
such that $\CE = \CE_{\vth}$ for $\vth \in (T^{\io\th,F})\wg$. 
Then 
\begin{equation*}
\x_{T,\CE}(su,v) = |Z_H(s)^F|\iv \sum_{\substack{x \in H^F \\
     x\iv sx \in T^{\io\th,F}}}Q_{xTx\iv}^{Z_G(s)}(u,v)\vth(x\iv sx). 
\end{equation*} 
\end{thm}
\par\bigskip
The proof of the theorem will be done by chasing 
Lusztig's arguments step by step.   We shall give 
an outline of the proof below.   First we need a lemma.

\begin{lem}  
Let $T \subset B$ be as in 2.2. For a semisimple element 
$s \in (G^{\io\th})^F$, there exists an open subset $\CU$ of 
$Z_G(s)^{\io\th}$ such that $e \in \CU$ and satisfying the 
following properties;
\begin{enumerate}
\item 
$g\CU g\iv = \CU$ for any $g \in Z_H(s)$, 
\item
$x \in \CU$ if and only if $x_s \in \CU$, 
\item
$F\CU = \CU$, 
\item
If $x \in \CU, g \in H, g\iv sxg \in B^{\io\th}$, 
then $g\iv x_sg \in B^{\io\th}$, and $g\iv sg \in B^{\io\th}$.
\item
If $x \in \CU, g \in H, g\iv sxg \in T^{\io\th}$, then 
$g\iv x_sg \in T^{\io\th}$ and $g\iv sg \in T^{\io\th}$.   
\end{enumerate}
\end{lem}

\begin{proof}
Put $\CU' = \{ x\in G \mid Z_G(x_s) \subset Z_G(s)\}$.  
Then $\CU'$ is stable by the conjugation action of $Z_H(s)$.
We define $\CU$ by $\CU = s\iv \CU' \cap Z_G(s)^{\io\th}$.
Then $\CU$ is an open subset of $Z_G(s)^{\io\th}$ containing 
$e$, and satisfies the condition (i), (ii), (iii). 
Assume that $x \in \CU, g \in H, g\iv sxg \in T^{\io\th}$. 
Then $g\iv sx_sg \in T^{\io\th} \subset T$.  Since $sx \in \CU'$, we have
$Z_G(g\iv sx_s g) \subset Z_G(g\iv sg)$, and so $T \subset Z_G(g\iv sg)$.
This implies that $g\iv sg \in T$ and $g\iv x_sg \in T$, thus (v) follows.
Next assume that $x \in \CU, g\in H, g\iv sxg \in B^{\io\th}$. 
Then $g\iv sx_s g \in B^{\io\th}$.
There exists $b \in B^{\th}$ such 
that $b\iv (g\iv sx_s g)b \in T^{\io\th}$.
By (v), we have $b\iv(g\iv sg)b \in T^{\io\th}$ and 
$b\iv(g\iv x_sg)b \in T^{\io\th}$.  
Hence $g\iv sg \in B^{\io\th}$ and  $g\iv x_sg \in B^{\io\th}$,
which proves (iv).  
\end{proof}

\para{2.6.}
For a fixed $m$, and an orbit $\CO_{\g}$ in 2.2, we define varieties

\begin{align*}
\CX_{m,\CU} &= \{ (sx,v) \in \CX_m^0 \mid x \in \CU\} \\
\wt\CX_{m,\CU} &= \{ (sx, v, gB^{\th}) \in \wt\CX^+_m
                          \mid x \in \CU\}, \\
\wt\CX_{m, \CU, \g} &= \{ (sx, v, gB^{\th}) \in \wt\CX_{m,\CU}
                            \mid g \in \CO_{\g} \}.  
\end{align*}
Let $p : \wt\CX_{m,\CU} \to H/B^{\th}$ be the projection on 
the third factor.  Then by Lemma 2.5 (iv), $\Im p$ is a subset of 
$\{ gB^{\th} \in H/B^{\th} \mid g\iv sg \in B^{\io\th}\}$, on which 
$Z_H(s)$ acts as a left multiplication, 
and the set of orbits is in bijection with $\vG$. 
Moreover, each orbit is open and closed, and $\wt\CX_{m,\CU,\g}$
coincides with the inverse image of $p$ of an orbit corresponding to $\g$.
It follows that 
\begin{equation*}
\tag{2.6.1}
\wt\CX_{m,\CU} = \coprod_{\g \in \vG}\wt\CX_{m,\CU,\g},
\end{equation*}
and $\wt\CX_{m,\CU,\g}$ is open and closed in $\wt\CX_{m,\CU}$.
The subvariety $\wt\CX'^+_{m,\g}$ of $\wt\CX'_{\g}$ is 
defined in a similar way as $\wt\CX^+_m$.
We define
\begin{equation*}
\wt\CX'_{m,\CU,\g} = \{ (x,v,gB^{\th}_{\g}) \in \wt\CX'^+_{m,\g}
              \mid x \in \CU\}.
\end{equation*}
Then we have
\par\medskip\noindent
(2.6.2) \ the map $(x,v, gB^{\th}_{\g}) \mapsto (sx, v, gx_{\g}B^{\th})$
gives an isomorphism $\wt\CX'_{m,\CU,\g} \isom \wt\CX_{m,\CU,\g}$.
\par\medskip
The proof is similar to [L3].
Put $\CX'_{m,\CU,\g} = \pi_{\g}(\wt\CX'_{m,\CU,\g})$, and 
we define $\CX'_{m,s\CU, \g}$ by 
$\CX'_{m,s\CU,\g} = 
  \{ (sx, v) \in s\CU \times V \mid (x,v) \in \CX'_{m,\CU, \g}\}$. 
Recall the map 
$\psi^{\bullet}_m : \wt\CY^{+,\bullet}_m \to \CY_m^0$
in 1.2. 
We define varieties
\begin{align*}
\CY_{m,\CU} &= \{ (sx,v) \in \CY_m^0 \mid x \in \CU\}, \\ 
\wt\CY^{\bullet}_{m,\CU} &= (\psi^{\bullet}_m)\iv (\CY_{m,\CU}), \\
\wt\CY^{\bullet}_{m,\CU,\g} &= \{ (sx, v, gT^{\th}) \in 
    \wt\CY^{\bullet}_{m,\CU} \mid g \in \CO_{\g}\}.
\end{align*}
Then the map $(x,v, gT^{\th}) \mapsto (x,v, gB^{\th})$ gives
rise to a vector bundle 
$\wt\CY^{\bullet}_{m,\CU} \to \pi\iv(\CY_{m,\CU})$ 
with fibre isomorphic to $U^n_2$, where 
$\pi\iv(\CY_{m,\CU})$ is an open dense subset of $\wt\CX_{m,\CU}$
and $U_2$ is the maximal unipotent subgroup of $SL_2$. 
Then $\CY_{m,\CU}$ is open dense in $\CX_{m,\CU}$.
As in [L3], we have
\begin{equation*}
\wt\CY^{\bullet}_{m,\CU} = \coprod_{\g \in \vG}\wt\CY^{\bullet}_{m,\CU,\g},
\end{equation*}
where $\wt\CY^{\bullet}_{m,\CU,\g}$ is a non-empty, open and closed 
subset of $\wt\CY^{\bullet}_{m,\CU}$. 
Put $\CY_{m,\CU,\g} = \psi^{\bullet}_m(\wt\CY^{\bullet}_{m,\CU, \g})$. 
Then for $\g, \g' \in \vG$, $\CY_{m,\CU,\g}$ and 
$\CY_{m,\CU,\g'}$ are either disjoint or coincide. 
They coincide if and only if $\g, \g'$ are in the same
$W_{H,m}$ orbit in $\vG$, where 
$W_{H,m} = \{ w\in W_H \mid w(M_m) = M_m\}$ acts on $\vG$
from the right.   Moreover, 
$\CY_{m,\CU} = \bigcup_{\g \in \vG}\CY_{m,\CU,\g}$ gives 
a decomposition of $\CY_{m,\CU}$ into irreducible components. 
\par
Let $\CY' = Z_G(s)^{\io\th}\reg \times V$, and put
\begin{equation*}
\wt\CY'^{\bullet}_{\g} = \{ (x,v, gT^{\th}_{\g}) \in 
     Z_G(s)^{\io\th}\reg \times V \times Z_H(s)/T^{\th}_{\g} 
   \mid g\iv xg \in T_{\g}^{\io\th}, g\iv v \in M_n^{\g} \}.
\end{equation*}
We define a map $\psi_{\g}: \wt\CY'^{\bullet}_{\g} \to \CY'$ by 
$(x,v, gT^{\th}_{\g}) \mapsto (x,v)$. 
The subvariety $\wt\CY'^{+,\bullet}_{m,\g}$ of 
$\wt\CY'^{\bullet}_{\g}$ is defined 
in a similar way as $\wt\CY^{+,\bullet}_m$ for $\wt\CY^{\bullet}$,  
We define 
\begin{equation*}
\wt\CY'_{m,\CU, \g} = \{(x,v, gT^{\th}_{\g}) \in \wt\CY'^{+,\bullet}_{m,\g} 
                     \mid x \in \CU\},
\end{equation*}
and put $\CY'_{m,\CU,\g} = \psi_{\g}(\wt\CY'_{m,\CU,\g})$.
Then the map $(x,v, gT^{\th}_{\g}) \mapsto (x,v, gB^{\th}_{\g})$
gives rise to a vector bundle  
$\wt\CY'_{m,\CU,\g} \to \pi_{\g}\iv(\CY'_{m,\CU,\g})$ with fibre 
isomorphic to $U_2^n$.
We also define 
\begin{equation*}
\CY'_{m, s\CU, \g} = \{ (sx,v) \in s\CU \times V 
                          \mid (x,v) \in \CY'_{m,\CU, \g} \}. 
\end{equation*}
For each $W_{H,m}$-orbit $Z$ in $\vG$, we define an open 
subset $\CV_Z$ of $\CY_{m,\CU}$ by 
\begin{equation*}
\CV_Z = \bigcap_{\g \in Z}(\CY'_{m,s\CU,\g} \cap \CY_{m,\CU,\g}).
\end{equation*}
Since $\CY'_{m,\CU,\g}$ is open dense in $\CX'_{m,\CU,\g}$,  
$\CY'_{m,s\CU,\g}$ is open dense in $\CX'_{m,s\CU,\g}$. 
Moreover, $\CY_{m,\CU,\g}$ is open dense in $\CX'_{m,s\CU,\g}$.
It follows that $\CY'_{m,s\CU,\g} \cap \CY_{m,\CU,\g}$ is open dense 
in $\CX'_{m,s\CU,\g}$.  In particular, $\CV_Z$ is an open dense 
subset of $\CY'_{m,s\CU,\g}$.  Here $\CY'_{m,\CU, \g}$  is 
open dense in ${\CY'}^0_{m,\g}$, where ${\CY'}^0_{m,\g}$ is a similar 
variety as $\CY^0_m$ defined for $Z_G(s)^{\io\th} \times V$.  
Hence $\CV_Z$ is smooth, irreducible, and by [SS, Lemma 3.3], we have
$\dim \CV_Z = \dim Z_G(s)^{\io\th} + 2m$, which is independent 
of $Z$.  Moreover, we have $F(\CV_Z) = \CV_{F(Z)}$.
Put $\CV = \bigcup_{Z}\CV_Z$, where $Z$ runs over all $W_{H,m}$-orbits
in $\vG$. Then 
\par\medskip\noindent
(2.6.3) \  
$\CV$ is an open dense smooth equidimensional subset of $\CY_{m,\CU}$
and $F(\CV) = \CV$.   Moreover, $\{\CV_Z\}$ gives the set of 
irreducible components in $\CV$.
\par\medskip
We have a commutative diagram

\begin{equation*}
\tag{2.6.4}
\begin{CD}
\wt\CY_m^{+,\bullet}|_{\CV} @<\wt\ve<< 
          \coprod_{\g \in \vG}(\wt\CY'^{+,\bullet}_{m,\g}|_{s\iv\CV})  \\  
@VVV    @VVV  \\
\CV  @<\ve<<   s\iv\CV,
\end{CD}
\end{equation*}
where
\begin{align*}
\wt\CY_m^{+,\bullet}|_{\CV} &= 
  \{ (sx,v, gT^{\th}) \in \wt\CY^{+,\bullet}_m \mid (sx,v) \in \CV\}, \\
\wt\CY'^{+,\bullet}_{m,\g}|_{s\iv\CV} &=
   \{ (x,v, gT_{\g}^{\th}) \in \wt\CY'^+_{m,\g} \mid 
                  (sx,v) \in \CV_Z\} \quad (\g \in Z), \\
s\iv\CV &= \{ (x,v) \in \CU \times V \mid (sx,v) \in \CV\},  
\end{align*}
and the map $\wt\ve$ is an isomorphism given by 
$(x,v, gT_{\g}^{\th}) \mapsto (sx, v, gx_{\g}T^{\th})$, 
$\ve$ is an isomorphism given by 
$(x,v) \mapsto (sx,v)$.
Moreover, the vertical maps are projections to the first and
the second factors.
\par
Recall the modified Frobenius map 
$F': \wt\CY_m^{+,\bullet} \to \wt\CY_m^{+,\bullet}$
given in 1.2, and consider the corresponding map $F'$ on 
$\wt\CY'^{+,\bullet}_{m,\g}$.  Then the varieties in the upper row 
in (2.6.4) are $F'$-stable, those in the lower row are 
$F$-stable, and all the 
maps are compatible with $F, F'$ actions. 
It follows that we have a canonical isomorphism of semisimple 
complexes (see (2.1.1))
\begin{equation*}
\tag{2.6.5}
\ve^*((j'_!(\psi_m^{\bullet})_*\wt\CE_m')|_{\CV}) \simeq 
\bigoplus_{\g \in \vG}((j'_{\g})_!(\psi^{\bullet}_{m,\g})_*
                \wt\CE'_{m,\g})|_{s\iv \CV}
\end{equation*}
and this isomorphism is compatible with the lifting of Frobenius maps
induced from $\vf_0: F^*\CE \isom\CE$, where 
$j'_{\g}, \psi^{\bullet}_{m,\g}, \wt\CE'_{m,\g}$ 
are the objects for $\wt\CY'^{+,\bullet}_{m,\g}$ corresponding 
to $j', \psi_m^{\bullet}, \wt\CE_m'$ for $\wt\CY_m^+$.
Let $K_m$ be as in 2.1, and let
$K_{m,\g} = 
(j_{\g})_!(\pi_{m,\g})_*\wt\CE'_{m,\g}[d_{m,\g}]$ the corresponding 
object on $\CX'_{m,\g}$, where $j_{\g}, \pi_{m,\g}, d_{m,\g}$ are defined 
similar to $j, \pi_m, d_m$. 
Then by (2.1.1), (2.1.3), and by the corresponding formulas for 
$K_{m,\g}$, the isomorphism in (2.6.5) can be regarded as an isomorphism 
\begin{equation*}
\tag{2.6.6}
\ve^*({}^pH^j(K_m)|_{\CV})[-\d] \simeq 
    \bigoplus_{\g \in \vG}{}^pH^j(K_{m,\g})|_{s\iv\CV}
\end{equation*}
for each $j$, 
where $\d = d_m - d_{m,\g} = \dim G^{\io\th} - \dim Z_G(s)^{\io\th}$
(see (2.1.4)).
We note that the isomorphism in (2.6.6) is 
the restriction of an isomorphism 
\begin{equation*}
\tag{2.6.7}
\ve^*({}^pH^j(K_m)|_{\CX_{m, \CU}})[-\d] \simeq \bigoplus_{\g \in \vG}
           {}^pH^j(K_{m,\g})|_{\CX'_{m, \CU, \g}},
\end{equation*}
where $\ve$ is an isomorphism from $\CX'_{m,\CU,\g}$ into 
$\CX_{m,\CU}$ given by $(x,v) \mapsto (sx,v)$.
In fact, since  
\begin{align*}
K_m|_{\CX_{m,\CU}} &\simeq j_!(\pi_m)_*(\a^*\CE|_{\wt\CX_{m,\CU}}), \\
K_{m,\g}|_{\CX'_{m,\CU,\g}} &\simeq 
  (j_{\g})_!(\pi_{m,\g})_*(\a_{\g}^*\CE|_{\wt\CX'_{m,\CU,\g}}),
\end{align*}
we obtain the isomorphism in (2.6.7) from (2.6.1) and (2.6.2),
which is clearly compatible with the isomorphism in (2.6.6).
\par
Now the complex ${}^pH^jK_m$ can be written as in (2.1.2),  and a similar
formula holds for ${}^pH^jK_{m,\g}$.  In particular, the right hand side 
of (2.6.7) can be given, as a semisimple complex, 
 by the intersection cohomologies obtained from local systems 
on the open dense smooth equidimensional subset 
of $s\iv\CV$.  Thus the isomorphism in (2.6.6) is uniquely extended 
to the isomorphism in (2.6.7), which is automatically compatible with 
the lifting of the Frobenius maps.
By (2.6.7), we obtain a relation of the stalks of the cohomology sheaves
for each $i$;
\begin{equation*}
\CH^{i- \d}_{(su,v)}{}^pH^jK_m \simeq 
 \bigoplus_{\g \in \vG}\CH^i_{(u,v)}{}^pH^jK_{m,\g}.
\end{equation*}
Under this isomorphism, the canonical map 
$\vf_m : F^*K_m \isom K_m$ on the left hand side induces an isomorphism
$\vf'_{m,\g}: F^*K_{m,\g} \isom K_{m,\g}$ for each $\g$ on the right hand
side.   Thus we have 
\begin{equation*}
\x_{{}^pH^jK_m,\vf_m}(su,v) = \sum_{\substack{\g \in \vG \\
               F(\g) = \g}}\x_{{}^pH^jK_{m,\g}, \vf'_{m,\g}}(u,v)
\end{equation*}  
for each $j$.
(Note that $\d = \dim G^{\io\th} - \dim Z_G(s)^{\io\th}$ is 
even since $\dim G^{\io\th} = 2n^2 -n$ by [SS, Lemma 1.9], and 
similarly $\dim Z_G(s)^{\io\th} = \sum_{i=1}^t(2n_i^2 - n_i)$ by 2.2.) 
Let $\vf_{m,\g}: F^*K_{m,\g} \isom K_{m,\g}$ be the canonical
isomorphism.  Then one can check that 
$\vf'_{m,\g} = \vth(x_{\g}\iv sx_{\g})\vf_{m,\g}$.
By using (2.1.5), we have
\begin{equation*}
\x_{T, \CE}(su,v) = \sum_{\substack{\g \in \vG \\ F(\g) = \g}}
         Q_{T_{\g}}^{Z_G(s)}(u,v)\vth(x_{\g}\iv sx_{\g}).          
\end{equation*}
Since $|\CO_{\g}^F| = |Z_H(s)^F||T^{\th F}||T_{\g}^{\th F}|\iv
      = |Z_H(s)^F|$ for an $F$-stable $\CO_{\g}$, we obtain 
the required formula.  The theorem is proved. 

\par\medskip

In view of Remark 1.8, we obtain the character formula for 
$\x^{\sym}_{T,\CE}$ as a corollary to Theorem 2.4.

\begin{cor}[Character formula for $\x^{\sym}_{T,\CE}$]  
Let the notations be the same as in Theorem 2.4.  Then we have 
\begin{equation*}
\x^{\sym}_{T,\CE}(su) = |Z_H(s)^F|\iv 
   \sum_{\substack{x \in H^F \\ x\iv sx \in T^{\io\th,F}}}
        Q^{Z_G(s), \sym}_{xTx\iv}(u)\vth(x\iv sx),
\end{equation*}
where $Q_{xTx\iv}^{Z_G(s),\sym}$ is the Green function 
$Q_{xTx\iv}^{\sym}$ for $Z_G(s)$. 
\end{cor}

\par\bigskip
\section{Orthogonality relations}

\para{3.1.}
Let $T$ be a $\th$-stable maximal torus of $G$ contained in 
a $\th$-stable Borel subgroup of $G$. Assume that $T$ is $F$-stable, 
and let $\CE$ be an $F$-stable tame local system on $T^{\io\th}$.
We consider the complex $K_{T,\CE}$ associated to the pair 
$(T,\CE)$ as in 1.3.  Let $\x_{T,\CE}$ and  
$Q_T$ be the functions defined in 1.5.  We also consider similar 
objects $\x^{\sym}_{T,\CE}, Q_T^{\sym}$ with respect to 
the symmetric space as in Remark 1.8.  In this 
section, we prove the orthogonality relations for these functions.
First we prepare a lemma.

\begin{lem}  
Let $T, T'$ be $\th$-stable maximal tori of $G$ as above (forgetting 
the $\Fq$-structure), and $\CE, \CE'$ 
tame local systems on $T^{\io\th}, {T'}^{\io\th}$.
Assume that $\CE'$ is a constant sheaf, and $\CE$ is a non-constant 
sheaf. 
\begin{enumerate}
\item
Let $K =K_{T,\CE}, K' = K_{T',\CE'}$  be  complexes  on 
$\CX = G^{\io\th} \times V$ with respect to 
$(T,\CE), (T',\CE')$.  Then we have
\begin{equation*}
\BH_c^i(\CX, K \otimes K') = 0 \quad 
   \text{ for all } i.
\end{equation*}
\item
Let $K = K^{\sym}_{T,\CE}, {K'} = {K'}^{\sym}_{T',\CE'}$
be complexes on $G^{\io\th}$ with respect to $(T,\CE), (T',\CE')$. 
Then we have
\begin{equation*}
\BH_c^i(G^{\io\th}, K \otimes K') = 0 \quad 
   \text{ for all } i.
\end{equation*}
\end{enumerate}
\end{lem}
 
\begin{proof}
We prove the proposition in a similar way as in the proof of 
Proposition 7.2 in [L3].  First consider the case (i).
We may replace $K_{T,\CE}$ by $\wt K_{B, \CE}$, 
and $K_{T',\CE'}$ by $\wt K_{B', \CE'}$, where 
$B$ (resp. $B'$) is a $\th$-stable Borel subgroup of 
$G$ containing $T$ (resp. $T'$).
We consider the fibre product 
$Z = \wt\CX \times_{\CX}\wt\CX'$, 
where $\wt\CX$ is the variety given in 1.4 attached to 
 $B, M_n$, and $\wt\CX'$ is a similar one attached to 
$B',M_n'$. 
Then $Z$ can be written as
\begin{equation*}
\begin{split}
Z = \{ (x,v, &\ gB^{\th}, g'{B'}^{\th}) \in G^{\io\th} \times V
              \times H/B^{\th} \times H/{B'}^{\th}  \\
     &\mid g\iv xg \in B^{\io\th}, {g'}\iv xg' \in {B'}^{\io\th},
            g\iv v \in M_n, {g'}\iv v \in M_n'\}. 
\end{split}
\end{equation*}
Note that $H$ acts on $Z$ so that the projections $Z \to \wt\CX$, 
$Z \to \wt\CX'$ are compatible with $H$-actions on those varieties. 
Let $\CL = \a^*\CE$ be a local system on $\wt\CX$, 
and $\CL'$ a similar local system on $\wt\CX'$.  
Since $K = \pi_*\CL$, and similarly for $K'$, up to shift, by the 
K\"unneth formula, we have 
\begin{equation*}
\tag{3.2.1}
\BH_c^i(\CX, K\otimes K') \simeq  H_c^i(Z, \CL\boxtimes \CL') 
\end{equation*}
up to shift.   Hence in oder to show the proposition, it is 
enough to see that the right hand side of (3.2.1) is equal to
zero for each $i$.
For each $H$-orbit $\CO$ of $H/B^{\th} \times H/{B'}^{\th}$, put
\begin{equation*}
Z_{\CO} = \{ (x,v, gB^{\th}, g'{B'}^{\th}) \in Z 
                \mid (gB^{\th}, g'{B'}^{\th}) \in \CO \}.
\end{equation*}
Then $Z = \coprod_{\CO}Z_{\CO}$ is a finite partition, and 
$Z_{\CO}$ is a locally closed subvariety of $Z$.   Hence 
by a cohomology exact sequence, it is enough to show that
$H_c^i(Z_{\CO}, \CL\boxtimes\CL') = 0$ for any $i$ and any $\CO$.
We define a morphism $f_{\CO} : Z_{\CO} \to \CO$ as the 
third and fourth projection.  Then by the Leray spectral sequence, 
we have
\begin{equation*}
H_c^i(\CO, R^j(f_{\CO})_!(\CL\boxtimes\CL')) \Rightarrow
           H_c^{i+j}(Z_{\CO}, \CL\boxtimes\CL').
\end{equation*}
Thus it is enough to show that 
$R^j(f_{\CO})_!(\CL\boxtimes\CL') = 0$ for any $j$, which
is equivalent to $H^j_c(f_{\CO}\iv(\xi), \CL\boxtimes\CL') = 0$
for any $j$ and any $\xi \in \CO$.  Since $\CL\boxtimes\CL'$ 
is an $H$-equivariant local system on $Z$,  
it is enough to show this for a single element $\xi \in \CO$. 
Hence we may choose $\xi = (B^{\th}, \w{B'}^{\th})$, 
where $\w \in H$ is such that $\w T'\w\iv = T$. (Note that 
$T$ and $T'$ are conjugate under $H$.)  Then 
$f_{\CO}\iv(\xi)$ is isomorphic to $Z(\w)$, where 
\begin{equation*}
Z(\w) = (B^{\io\th} \cap \w{B'}^{\io\th}\w\iv) \times (M_n \cap \w(M_n')).
\end{equation*}
Here we have
\begin{equation*}
B^{\io\th} \cap \w {B'}^{\io\th}\w\iv 
  = \{ su \in B \mid s \in T^{\io\th}, 
           u \in U \cap \w U' \w\iv, \th(u) = su\iv s\iv\}. 
\end{equation*}
Then the map $t_1*u_1 \mapsto \th(t_1\iv)t_1\cdot (t_1\iv u_1t_1)$
gives an isomorphism 
\begin{equation*}
T\times^{T^{\th}}(U \cap \w U'\w\iv)^{\io\th} 
\isom B^{\io\th} \cap \w {B'}^{\io\th}\w\iv.
\end{equation*}
It follows that the projection 
$B^{\io\th} \cap \w{B'}^{\io\th}\w\iv \to T^{\io\th} \simeq T/T^{\th}$ 
gives
rise to a vector bundle over $T^{\io\th}$. 
Hence $Z(\w)$ is also a vector bundle over $T^{\io\th}$, of rank 
say $d$. 
If we denote by 
$a : T^{\io\th} \to {T'}^{\io\th} = \w\iv T^{\io\th}\w$,  
we have
\begin{equation*}
H_c^{i +2d}(Z(\w), \CL\boxtimes \CL') \simeq 
      H_c^i(T^{\io\th}, \CE\otimes a^*\CE') \simeq H_c^i(T^{\io\th},\CE)
\end{equation*}
since $\CE'$ is a constant sheaf. 
Hence we have only to show that $H_c^i(T^{\io\th}, \CE) = 0$ 
for any $i$. But this certainly holds since $\CE$ is 
a non-constant tame local system on $T^{\io\th}$.  Thus (i) 
is proved.  (ii) is proved in a similar way, just ignoring 
the vector space part. 
\end{proof}

\para{3.3.}
For $\th$-stable maximal tori $T, T'$ conjugate under $H$, put
\begin{align*}
N_H(T^{\th}, {T'}^{\th}) &= 
   \{ n \in H \mid n\iv T^{\th}n = {T'}^{\th}\}, \\
N_H(T^{\io\th}, {T'}^{\io\th}) &= 
   \{ n \in H \mid n\iv T^{\io\th}n = {T'}^{\io\th}\}.
\end{align*} 
Since $T, T'$ are $H$-conjugate, 
there exists $h \in H$ such that $h\iv Th = T'$.
It follows that $N_H(T^{\th},{T'}^{\th}) = 
      N_H(T^{\th})h = hN_H({T'}^{\th})$, and 
$N_H(T^{\io\th},{T'}^{\io\th}) = 
      N_H(T^{\io\th})h = hN_H({T'}^{\io\th})$.  
Since one can check that $N_H(T^{\th}) \subset N_H(T^{\io\th})$, 
we have 
\begin{equation*}
\tag{3.3.1}
N_H(T^{\th}, {T'}^{\th}) \subset N_H(T^{\io\th}, {T'}^{\io\th}).
\end{equation*}

The following orthogonality relations are an analogy of 
Theorem 9.2 and Theorem 9.3 in [L3].
(In the following, we use the notation as in Theorem 2.4, 
 $T^{\th,F} = T^{\th} \cap T^F$
and $T^{\io\th, F} = T^{\io\th} \cap T^F$.)   

\begin{thm} [Orthogonality relations for $\x_{T,\CE}$]   
Assume that $T, T'$ are $F$-stable, $\th$-stable maximal tori in $G$
as in 3.1. Let $\CE = \CE_{\vth}, \CE' = \CE_{\vth'}$ be tame local systems 
on $T^{\io\th}, {T'}^{\io\th}$ with 
$\vth \in (T^{\io\th, F})\wg, \vth' \in ({T'}^{\io\th,F})\wg$. 
 Then we have
\begin{equation*}
\tag{3.4.1}
\begin{split}
|H^F|\iv&\sum_{(x,v) \in \CX^F}\x_{T,\CE}(x,v)\x_{T',\CE'}(x,v)  \\
    &= |T^{\th,F}|\iv|T'^{\th,F}|\iv\sum_{\substack{ n \in 
         N_H(T^{\th}, {T'}^{\th})^F \\ t \in T^{\io\th, F} }}
             \vth(t)\vth'(n\iv tn).
\end{split}
\end{equation*}
(Note that $n\iv tn \in {T'}^{\io\th,F}$ by (3.3.1)).
\end{thm}

\begin{thm}[Orthogonality relations for Green functions]  
\begin{equation*}
\tag{3.5.1}
|H^F|\iv\sum_{(u,v) \in \CX\uni^F} Q_T(u,v)Q_{T'}(u,v) 
     = \frac{|N_H(T^{\th}, {T'}^{\th})^F|}
              {|T^{\th, F}||{T'}^{\th, F}|}. 
\end{equation*}
\end{thm}

\para{3.6.}
In view of the decomposition in (2.2.1), the definition 
of $\wt\CX, \CX, \x_{T,\CE}, Q_T$, etc.  makes sense if we replace 
$G$ by $Z_G(s)$ for a semisimple element $s \in (G^{\io\th})^F$, 
and $\CX = G^{\io\th} \times V$ by $Z_G(s)^{\io\th} \times V$.
All the results in [SS] and the results in the previous sections 
can be extended to the case of $Z_G(s)$. 
Thus Theorem 3.4 and Theorem 3.5 are formulated for this 
general setting.  In order to make the inductive argument smoothly, 
we shall prove Theorem 3.4 and Theorem 3.5 simultaneously 
under this setting. 
\par
First we note that (3.4.1) holds when $\vth'$ is the trivial character 
and $\vth$ is a non-trivial character.  In fact, by 
the Grothendieck's trace formula for 
the Frobenius map, the left hand side of (3.4.1) coincides with 
\begin{equation*}
\sum_i (-1)^i\Tr(F^*, \BH_c^i(\CX, K_{T,\CE}\otimes K_{T',\CE'})),
\end{equation*}
where $F^*$ is an isomorphism induced from 
$\vf: F^*K_{T,\CE} \isom K_{T,\CE}$ and a similar isomorphism 
for $K_{T',\CE'}$.   Then by Lemma 3.2 (i), we see that the left 
hand side of (3.4.1) is equal to zero.  On the other hand, 
the right hand side of (3.4.1) is equal to zero by the orthogonality 
relations for irreducible characters of $T^{\io\th, F}$.   
Hence the assertion holds.

\para{3.7.}  We shall verify the equality (3.5.1) in the special
case where $n = 1$, namely in the case where $G = GL_2$ and 
$H = SL_2$.  In this case the $H$-orbits in 
$\CX\uni = G^{\io\th}\uni \times V$ are 
parametrized by $\CP_{n,2} = \{ (1;-), (-;1)\}$.
The $H$-orbit corresponding to $(-;1)$ (resp. $(1;-)$ ) 
is represented by $z_0 = (e,0) \in \CX^F\uni$ 
(resp. $z_1 = (e,v) \in \CX^F\uni$ with $v \ne 0$), where 
$e \in G^{\io\th}$ is the unit element in $G$.
Then $\pi_1\iv(z_0) = H/B^{\th}$, 
and $\pi_1\iv(z_1) = \{ B^{\th}_1 \}$, 
where $B_1^{\th}$ is the stabilizer in $H$ of a line determined by
$v \in V$.  Let $T_0$ be the maximal torus of $G$ consisting of 
diagonal matrices.  The set of $F$-stable  maximal tori of $G$ 
conjugate to $T_0$ under $H$ is
parametrized by $W_2 \simeq \BZ/2\BZ = \{\pm 1\}$.  
$T_0$ corresponds to $1 \in W_2$, and let  
$T_1$ be a maximal torus corresponding to $-1 \in W_2$.  
Then $|T_0^{\th, F}| = q-1$, $|T_1^{\th, F}| = q+1$, and  
we have
\begin{equation*}
\begin{cases}
Q_{T_0}(z_0) = (-1)^{\dim \CX}(q + 1), \\
Q_{T_1}(z_0) = (-1)^{\dim \CX}(-q +1), \\
Q_{T_0}(z_1) = (-1)^{\dim\CX}, \\
Q_{T_1}(z_1) = (-1)^{\dim\CX},
\end{cases}
\end{equation*}
where $\dim \CX = 2n^2 + n = 3$.  
If we put $\lp Q_T, Q_{T'}\rp_{\ex} = |H^F|\iv
      \sum_{z \in \CX\uni^F}Q_T(z)Q_{T'}(z)$, 
one can easily compute that
\begin{equation*}
\begin{cases}
\lp Q_{T_0}, Q_{T_1}\rp_{\ex} = 0 , \\
\lp Q_{T_0}, Q_{T_0}\rp_{\ex} = |H^F|\iv (2q^2 + 2q), \\
\lp Q_{T_1}, Q_{T_1}\rp_{\ex} = |H^F|\iv (2q^2 - 2q).
\end{cases}
\end{equation*}
Since $|N_H(T^{\th})^F|/|T^{\th, F}| = 2$ for $T = T_0, T_1$, 
and $|H^F| = q(q^2-1)$, 
the equality (3.5.1) holds in this case.
\par 
Next we assume that $G$ is an  $r$ product of $GL_2$, on which
$F$ acts transitively.  Then the formula (3.5.1) is reduced to 
the case where $H = SL_2$ with $F$ replaced by $F^r$.  Hence 
it holds also in this case.  Thus (3.5.1) holds in the case 
where $G$ is a product of $GL_2$ on which $F$ acts as a permutation
of factors.

\para{3.8.}
In this subsection, we show that Theorem 3.4 holds for $G$ 
under the assumption that Theorem 3.5 holds for the subgroup
$Z_G(s)$ for any semisimple element $s \in (G^{\io\th})^F$.
Let $G^{\io\th}_{\ss}$ be the set of semisimple elements in 
$G^{\io\th}$.
By making use of the character formula (Theorem 2.4), we have
\begin{equation*}
\begin{split}
|H^F|\iv&\sum_{(x,v) \in \CX^F}\x_{T,\CE}(x,v)\x_{T'\CE'}(x,v)  \\
&= |H^F|\iv\sum_{\substack{ s \in (G^{\io\th}_{\ss})^F \\
        x,x' \in H^F \\
   x\iv sx \in T^{\io\th, F} \\
   {x'}\iv sx' \in {T'}^{\io\th, F} }}f(s,x,x')|Z_H(s)^F|^{-2}
          \vth(x\iv sx)\vth'({x'}\iv sx'),
\end{split}
\end{equation*}
where 
\begin{equation*}
f(s,x,x') = \sum_{(u,v) \in (\CX\uni^{Z_G(s)})^F}
       Q_{xTx\iv}^{Z_G(s)}(u,v)Q_{x'T'{x'}\iv}^{Z_G(s)}(u,v),
\end{equation*}
($\CX\uni^{Z_G(s)}$ denotes the object $\CX\uni$ defined for 
$Z_G(s)$ instead of $G$).  
By applying Theorem 3.5 for $Z_G(s)$, we see that 
\begin{equation*}
f(s,x,x') = |Z_H(s)^F||T^{\th, F}|\iv|{T'}^{\th, F}|\iv
    \sharp\{ n \in Z_H(s)^F \mid n\iv xT^{\th}x\iv n = 
                  x' {T'}^{\th}{x'}\iv \}.
\end{equation*}
It follows that the previous sum is equal to
\begin{equation*}
\begin{split}
|H^F|\iv &|T^{\th, F}|\iv |{T'}^{\th, F}|\iv  \\
  &\times \sum_{\substack{ s \in (G^{\io\th}_{\ss})^F \\
      x,x' \in H^F \\
      x\iv sx \in T^{\io\th, F} \\
      {x'}\iv sx' \in {T'}^{\io\th, F} }}
         |Z_H(s)^F|\iv \sum_{\substack{ n \in Z_H(s)^F \\
                   n\iv xT^{\th}x\iv n = x' {T'}^{\th}{x'}\iv }}
                        \vth(x\iv sx)\vth'({x'}\iv sx').
\end{split}
\end{equation*}
Now put $t = x\iv sx \in T^{\io\th,F}$ and $y = x\iv nx'$.
We have $y \in H^F$ with $y\iv T^{\th}y = {T'}^{\th}$, and so
$y\iv T^{\io\th}y = {T'}^{\io\th}$ by (3.3.1).
Then the condition for $x'$ is given by $x' \in (xZ_H(t)y)^F$.
Under this change of variables, the above sum can be rewritten as
\begin{equation*}
|H^F|\iv |T^{\th, F}|\iv |{T'}^{\th, F}|\iv\sum_{\substack{ x \in H^F\\
          t \in T^{\io\th, F}}} |Z_H(t)^F|\iv 
             \sum_{\substack{ y \in N_H(T^{\th}, {T'}^{\th})^F \\
                        x' \in (xZ_H(t)y)^F }}
             \vth(t)\vth'(y\iv ty)
\end{equation*}
which is equal to 
\begin{equation*}
|T^{\th, F}|\iv|{T'}^{\th, F}|\iv\sum_{t \in T^{\io\th, F}}
        \sum_{y \in N_H(T^{\th},{T'}^{\th})^F}
             \vth(t)\vth'(y\iv ty).
\end{equation*}
Thus our assertion holds. 

\para{3.9.} 
We shall show that Theorem 3.5 holds for $G$ under the assumption 
that it holds for $Z_G(s)$ if $s$ is not central. Hence Theorem 3.4 
holds for such groups $Z_G(s)$ by 3.8.  Put
\begin{align*}
A_1 &= |H^F|\iv \sum_{(u,v) \in \CX\uni^F}
           Q_T(u,v)Q_{T'}(u,v), \\
A_2 &= |T^{\th, F}|\iv|{T'}^{\th, F}|\iv |N_H(T^{\th}, {T'}^{\th})^F|. 
\end{align*}
By making use of a part of the arguments in 3.8 (which can be applied
to the case where $s \notin Z(G)^{\io\th, F}$), we see that 

\begin{equation*}
\tag{3.9.1}
\begin{split}
|H^F|\iv&\sum_{(x,v) \in \CX^F}
         \x_{T,\CE}(x,v)\x_{T',\CE'}(x,v) - 
             A_1\vT  \\
&= |T^{\th, F}|\iv|{T'}^{\th, F}|\iv \sum_{t \in T^{\io\th, F}}
            \sum_{y \in N_H(T^{\th}, {T'}^{\th})^F}
  \vth(t)\vth'(y\iv ty) - A_2\vT
\end{split}
\end{equation*}
with $\vT = \sum_{s \in Z(G)^{\io\th,F}}\vth(s)\vth'(s)$.
This formula holds for any $\vth \in (T^{\io\th, F})\wg$, 
$\vth' \in ({T'}^{\io\th, F})\wg$.  In the case where 
$G$ is a product of $GL_2$, Theorem 3.5 is verified in 3.7.  
  So we may assume that there exists a factor of $G$ of the form
$GL_{2n}$ with $n \ge 2$. 
Then one can find a linear character $\vth$ of $T^{\io\th, F}$
such that $\vth|_{Z(G)^{\io\th, F}} = \id$ and that 
$\vth \ne \id$.  We choose $\vth'$ the identity character of 
${T'}^{\io\th, F}$.  Then by 3.6, the first term of the left hand 
side of (3.9.1) coincides with the first term of the right hand side.
Since $\vT = |Z(G)^{\io\th,F}| \ne 0$ in this case, 
we obtain $A_1 = A_2$ as asserted.

\para{3.10.}
We are now ready to prove Theorem 3.4 and Theorem 3.5.  
First note that Theorem 3.5 holds for $G$ in the case 
where $G$ is a product of $GL_2$ by 3.7.  Thus Theorem 3.4
holds for such $G$ by 3.9.  Next we consider the general $G$.  
By induction on the semisimple rank of $G$,  
we may assume that Theorem 3.4 and 3.5 hold for $Z_G(s)$ with 
$s$ is not central. Hence Theorem 3.5 holds for $G$ by 3.9,
and so Theorem 3.4 holds for $G$ by 3.8.    
This completes the proof of Theorem 3.4 and Theorem 3.5.
\par\medskip
By applying a similar argument as in the proof of Theorem 3.4 
and Theorem 3.5, one can prove the orthogonality 
relations in the case of symmetric spaces.

\begin{thm}[Orthogonality relations for $\x^{\sym}_{T,\CE}$]
Let $\CE = \CE_{\vth}, \CE' = \CE_{\vth'}$ with 
$\vth \in (T^{\io\th, F})\wg, \vth' \in ({T'}^{\io\th, F})\wg$.
Then we have
\begin{equation*}
\begin{split}
|H^F|\iv&\sum_{x \in G^{\io\th, F}}
        \x^{\sym}_{T,\CE}(x)\x^{\sym}_{T',\CE'}(x) \\
  &= |T^{\io\th, F}|\iv|{T'}^{\io\th, F}|\iv q^{-2r} 
            \sum_{\substack{n \in N_H(T^{\io\th}, {T'}^{\io\th})^F \\ 
          t \in T^{\io\th, F}}}\vth(t)\vth'(n\iv tn),
\end{split}
\end{equation*}
where $r$ is the rank of $H$. 
\end{thm}

\begin{thm}[Orthogonality relations for Green functions $Q^{\sym}_T$]
\begin{equation*}
\tag{3.12.1}
|H^F|\iv\sum_{u \in G^{\io\th}\uni}
        Q_T^{\sym}(u)Q_{T'}^{\sym}(u) = 
          q^{-2r} \frac{|N_H(T^{\io\th}, {T'}^{\io\th})^F|}
                  {|T^{\io\th, F}||{T'}^{\io\th, F}|}
\end{equation*}
\end{thm}

\begin{proof}
As in the proof of Theorem 3.4 and Theorem 3.5, we prove 
these two theorems simultaneously for the groups of the form  
$Z_G(s)$.  
First we consider the case where $G = GL_2$.  In this case, 
$G^{\io\th}\uni = \{ e\}$, and $T_0^{\io\th}$ coincides with the center 
of $G$.  We have $\CW = \{ 1\}$.  Let $u = e$.  Then 
$\pi_1\iv(u) = \CB^{\th}$, where 
$\pi_1 : \wt G^{\io\th}\uni \to G^{\io\th}\uni$ is as in 
[SS, 1.10].
Thus $Q^{\sym}_{T_0}(u) = (-1)^{\dim \wt G^{\io\th}}(q+1)$.
We see that the left hand side of (3.12.1) for $T = T_0$ is equal to 
\begin{equation*}
|H^F|\iv(q+1)^2 = q\iv(q+1)(q-1)\iv.
\end{equation*}

On the other hand, 
\begin{equation*}
N_H(T_0^{\io\th})^F/|T_0^{\io\th F}|^2 = |H^F|/(q-1)^2 = q(q+1)(q-1)\iv.
\end{equation*}
Since $r = 1$, we obtain  the equality in (3.12.1). 
By a similar argument as in 3.7, 
this implies that Theorem 3.12 holds for $G$, where $G$ is 
a product of $GL_2$.  Then the arguments in 3.6 $\sim$ 3.10 are 
applied to our situation, by using Lemma 3.2 (ii) instead of (i), 
and by the character formula (Corollary 2.7). 
This proves the theorems.  
\end{proof}

\par\bigskip
\section{A purity result}  

\para{4.1.}
For an $\th$-stable Borel subgroup of $G$, we denote the map
$\pi_1: \wt\CX\uni \to \CX\uni$ given in 1.4 as 
$\pi_{1,B} : \wt\CX_{\unip,B} \to \CX\uni$ to indicate the dependence
on $B$.
Let $K_1 = (\pi_{1,B_0})_*\Ql[\dim \CX\uni]$.
By the Springer correspondence ([SS, Theorem 5.4]), we have 
\begin{equation*}
K_1 \simeq \bigoplus_{\Bmu \in \CP_{n,2}} \r_{\Bmu}
              \otimes A_{\Bmu},
\end{equation*}  
where $A_{\Bmu} = \IC(\ol\CO_{\Bmu}, \Ql)[\dim \CO_{\Bmu}]$, 
and $\r_{\Bmu} \simeq \Hom (A_{\Bmu}, K_1)$ 
is the irreducible $W_n$-module corresponding 
to $\Bmu$ by [SS, Theorem 7.1].  
Since $\wt\CX_{\unip,B_0}$ has a natural $\Fq$-structure, we have 
an isomorphism $\vf_1 : F^*K_1 \isom K_1$.
Since each $H$-orbit $\CO_{\Bmu}$ is $F$-stable, 
we have $F^*A_{\Bmu} \simeq A_{\Bmu}$.  Hence
$\vf_1$ induces 
an isomorphism $F^*(\r_{\Bmu} \otimes A_{\Bmu}) \isom 
(\r_{\Bmu}\otimes A_{\Bmu})$.  It follows that there exists a unique 
isomorphism $\vf_{\Bmu}: F^*A_{\Bmu} \isom A_{\Bmu}$ such that 
$\vf_1 = \sum_{\Bmu}\s_{\Bmu}\otimes \vf_{\Bmu}$, where 
$\s_{\Bmu}$ is the identity map on $\r_{\Bmu}$.  
Let $\f_{\Bmu}: F^*A_{\Bmu} \isom A_{\Bmu}$ be the natural
isomorphism induced form the $\Fq$-structure of $\CO_{\Bmu}$.
Since $A_{\Bmu}$ is a simple perverse sheaf, 
$\vf_{\Bmu}$ coincides with $\f_{\Bmu}$ up to scalar.
Let $d_{\Bmu} = (\dim \CX\uni - \dim \CO_{\Bmu})/2$.
We note that 
\par\medskip\noindent
(4.1.1) \  $\vf_{\Bmu} = q^{d_{\Bmu}}\f_{\Bmu}$.  In particular, 
the map $\vf_{\Bmu}$ gives a scalar multiplication $q^{d_{\Bmu}}$
on $A_{\Bmu}|_{\CO_{\Bmu}}$. 
\par\medskip 
In fact, for $z \in \CX\uni$, we have
$\CH^i_zK_1 \simeq H^{i + \dim \CX\uni}(\CB_z, \Ql)$, 
where $\CB_z$ is the closed subvariety of $\CB = H/B^{\th}$
isomorphic to $\pi_{1,B}\iv(z)$.
For $z \in \CO_{\Bmu}$, we have ([SS, (5.4.3)]) 
\begin{equation*}
 H^{2d_{\Bmu}}(\CB_z, \Ql) \simeq \r_{\Bmu} \otimes 
       \CH^0_z \IC(\ol\CO_{\Bmu},\Ql) \simeq \r_{\Bmu}. 
\end{equation*} 
Here $d_{\Bmu} = \dim \CB_z$, and $H^{2d_{\Bmu}}(\CB_z, \Ql)$
is an irreducible $W_n$-module.  Since the Frobenius action on 
$H^{2d_{\Bmu}}(\CB_z, \Ql)$ commutes with the $W_n$ action, we see that 
the Frobenius map acts on $H^{2d_{\Bmu}}(\CB_z, \Ql)$ 
as a scalar multiplication.  In particular, all the irreducible 
components in $\CB_z$ are $F$-stable, and this scalar is given by 
$q^{d_{\Bmu}}$.  
It follows that $\vf_{\Bmu}$ acts as a scalar multiplication 
$q^{d_{\Bmu}}$ on $\CH^0_z \IC(\ol\CO_{\Bmu},\Ql) \simeq \Ql$.
Since $\f_{\Bmu}$ gives the identity map on this space, we obtain 
(4.1.1).
\par
We choose a $\th$-stable Borel subgroup $B$ containing $T = T_w$
for $w \in W_n$, and consider the complex 
$K_w = (\pi_{1,B})_*\Ql[\dim \CX\uni]$.
By 1.3, we have a canonical isomorphism $\vf_w: F^*K_w \isom K_w$.
Note that, by definition, we have 
$Q_{T_w} = (-1)^{\dim \CX - \dim \CX\uni}\x_{K_w, \vf_w}$.
In view of (1.7.3), we have 
\begin{equation*}
\tag{4.1.2}
Q_{T_w} = (-1)^{\dim \CX - \dim \CX\uni}
   \sum_{\Bmu \in \CP_{n,2}}\x^{\Bmu}(w)\x_{A_{\Bmu},\vf_{\Bmu}},
\end{equation*}
where $\x^{\Bmu}$ is the irreducible character of $W_n$
corresponding to $\r_{\Bmu}$.
(Note that $\x^{\Bmu}(w) = \x^{\Bmu}(w\iv)$, and that 
the effect of $F^*$ on $\th_w$ is ignored since 
$F$ acts trivially on $W_n$.)

\para{4.2.}
We shall consider a similar setting for $G^{\io\th}\uni$.
Let $K_1 = (\pi_{1, B_0})_*\Ql[\dim G^{\io\th}\uni]$ with respect 
to $\pi_1: \wt G^{\io\th}\uni \to G^{\io\th}\uni$. 
By [SS, Remark 7.5], we have 
\begin{equation*}
K_1 \simeq H^{\bullet}(\BP_1^n)\otimes  
         \bigoplus_{\mu \in \CP_n}\r_{\mu}\otimes A_{\mu},
\end{equation*}
where $A_{\mu} = \IC(\ol \CO_{\mu},\Ql)[\dim \CO_{\mu}]$ with respect 
to the $H$-orbit $\CO_{\mu}$ in $G^{\io\th}\uni$, and 
$\r_{\mu}$ is
the irreducible $S_n$-module such that 
$H^{\bullet}(\BP_1^n)\otimes \r_{\mu} \simeq \Hom(A_{\mu}, K_1)$. 
The natural $\Fq$-structure on $\wt G^{\io\th}\uni$ induces 
an isomorphism $\vf_1 : F^*K_1 \isom K_1$, and it determines an isomorphism 
$\vf_{\mu} : F^*A_{\mu} \isom A_{\mu}$ for each $\mu$ such that 
$\vf_1 = (q+1)^n\sum_{\mu \in \CP_n}\s_{\mu} \otimes \vf_{\mu}$  
with $\s_{\mu}$
the identity map on $\r_{\mu}$.  As in (4.1.1) we see that 
\par\medskip\noindent
(4.2.1) \ $\vf_{\mu} = q^{d_{\mu}}\f_{\mu}$, where 
$d_{\mu} = (\dim G^{\io\th}\uni - \dim \CO_{\mu})/2$, and 
$\f_{\mu}$ is the isomorphism $F^*A_{\mu} \isom A_{\mu}$ 
induced from the natural $\Fq$-structure on $\CO_{\mu}$.  In particular, 
the map $\vf_{\mu}$ gives a scalar multiplication $q^{d_{\mu}}$ 
on $A_{\mu}|_{\CO_{\mu}}$.
\par\medskip
We consider the complex $K_w = (\pi_{1, B})_*\Ql[\dim G^{\io\th}\uni]$
on $G^{\io\th}\uni$ 
with respect to a $\th$-stable Borel subgroup $B$ containing $T = T_w$
for $w \in S_n$.  By Remark 1.8, one can define an isomorphism 
$\vf_w: F^*K_w \isom K_w$.  By definition, 
$Q_{T_w}^{\sym} = 
(-1)^{\dim G^{\io\th} - \dim G^{\io\th}\uni}\x_{K_w, \vf_w}$.  
A similar discussion as in 1.7 can be applied also for this case.  Then 
the map $\vf_w$ is described by the action of $S_n$ on each 
factor $\r_{\mu}$ and by the map $\vf_{\mu}$.  But contrast to the 
case for $\CX\uni$, the Frobenius action on $K_w$ 
permutes the factors on $\BP_1^n$ through the permutation $w$, and 
induces an action $F^*$ on $H^{\bullet}(\BP_1^n)$.  
Here we define a polynomial $\Psi_{\nu}(t)$ for each partition 
$\nu = (\nu_1, \dots, \nu_k)$ of $n$ by 
$\Psi_{\nu}(t) = \prod_{i=1}^k(t^{\nu_i} + 1)$, and define, for 
each $w \in S_n$,  $\Psi_w(t) = \Psi_{\nu}(t)$ 
if $w$ is of type $\nu$.  Then one sees that 
$\x_{F^*, H^{\bullet}(\BP_1^n)} = \Psi_w(q)$.
Thus as in (4.1.2), we have 
\begin{equation*}
\tag{4.2.2}
Q_{T_w}^{\sym} = (-1)^{\dim G^{\io\th} - \dim G^{\io\th}\uni}
      \sum_{\mu \in \CP_n}\Psi_w(q)\x^{\mu}(w)\x_{A_{\mu}, \vf_{\mu}},
\end{equation*}
where $\x^{\mu}$ is the irreducible character of $S_n$ corresponding 
to $\r_{\mu}$. 
\para{4.3.}
We now consider the Frobenius action on $A_{\Bmu}$ more precisely.
In general, let $K$ be a complex on a variety $X$ defined over 
$\Fq$ such that $F^*K \simeq K$.  An isomorphism $\f : F^*K \isom K$
is said to be pointwise pure if the eigenvalues of $\f$ on 
$\CH_x^iK$ are algebraic integers all of whose complex conjugate
have absolute value $q^{i/2}$ for any $x \in X^F$ and for any $i$.
Returning to our setting, we have the following result.
\begin{prop}  
Let $A = \IC(\ol\CO, \Ql)$ and $\f : F^*A \isom A$ be the 
isomorphism induced from the Frobenius map $F$ on $\CO$, where
$\CO$ is an $H$-orbit in $\CX\uni$. 
Then $\f$ is pointwise pure. 
\end{prop}

\begin{proof}
Since $G^{\io\th}\uni \times V$ is isomorphic to 
$\Fg^{-\th}\nil \times V$, compatible with $H$-action 
(see [SS, 1.7]), the $H$-orbit $\CO$ is isomorphic to the corresponding 
$H$-orbit in $\Fg^{-\th}\nil \times V$.  So we may pass to the 
Lie algebra case, and consider an $H$-orbit $\CO$ in 
$\Fg^{-\th}\nil \times V$.   
In [AH, Corrigendum, Proposition 3], Achar-Henderson constructed 
a transversal slice for a $G$-orbit of $\Fg\nil \times V$ with 
a suitable one parameter subgroup action.  Modifying their argument, 
we will construct a transversal slice for $\CO$.  
Recall a closed subgroup $A \simeq GL_n$ 
of $G$ given in [SS, 1.4].  Put $\Fa = \Lie A$ and 
$\Fa' = \th(\Fa)$. 
Thus $\Fa = \Fg\Fl(M_n)$, $\Fa'= \Fg\Fl(M_n')$ with 
$V = M_n \oplus M_n'$.
Assume that $\CO = \CO_{\Bmu}$ 
with $\Bmu = (\mu^{(1)}, \mu^{(2)})$.
Take $(x,v) \in \CO_{\Bmu}$.    
We may assume that 
$x = y -\th(y)$ with $y \in \Fa$
and $v \in M_n$.   
Put $\nu =  \mu^{(1)} + \mu^{(2)}$ with $\nu = 
(\nu_1, \nu_2, \dots, \nu_{m})$.
Then $(y,v) \in \Fa \times M_n$ is of type $\Bmu$.
We fix a normal basis 
$\{ v_{i,j} \mid 1 \le i \le m, 1 \le j \le \nu_i\}$ of 
$M_n$ with respect to $(y,v)$, 
and a basis $\{ v'_{i,j}\}$ of $M_n'$ as in [SS, 7.2].
(Here we may choose $v = \sum_{i=1}^mv_{i,\mu_i^{(1)}}$ as in 
the original setting in [AH].)
Following [AH], we define a subspace $U$ of $\Fg$ with basis 
\begin{equation*}
\begin{split}
\{ z_{i_1,i_2,s}, &z_{i^*_1, i_2,s}, z_{i_1, i_2^*,s}, 
     z_{i_1^*,i_2^*,s}  \\
   &\mid 1 \le i_1, i_2 \le m, 
    \max\{ 0, \nu_{i_1} - \nu_{i_2}\} \le s \le \nu_{i_1} - 1 \}.
\end{split}
\end{equation*}
Here 
$z_{i_1,i_2,s}$ (resp. $z_{i^*_1,i_2,s}, z_{i_1,i_2^*, s}, 
z_{i_1^*,i_2^*,s}$) is the map defined by sending the basis 
$v_{i_2, 1}$ to $v_{i_1, s+1}$ (resp. $v_{i_2,1}$ to 
$v'_{i_1,\nu_{i_1}-s}$, 
$v'_{i_2,\nu_{i_2}}$ to $v_{i_1,s+1}$, 
$v'_{i_2,\nu_{i_2}}$ to $v'_{i_1,\nu_{i_1}-s}$), 
and sending all other basis to zero.  
Then $U$ is a $\th$-stable subspace since 
$\th(z_{i_1,i_2,s}) = \pm z_{i^*_1,i^*_2,s}$, 
$\th(z_{i^*_1,i_2,s}) = \pm z_{i_1,i_2^*,s}$. 
Achar-Henderson proved that 
\begin{equation*}
\tag{4.4.1}
[\Fg, x] \oplus U = \Fg.
\end{equation*}
Since $U$ is $\th$-stable, we have $U = U^{\th} \oplus U^{-\th}$.
Since $x \in \Fg^{-\th}$, we have $[\Fg^{\th},x] \subset \Fg^{-\th}$, 
$[\Fg^{-\th},x] \subset \Fg^{\th}$, and so 
$[\Fg,x] = [\Fg^{\th},x] \oplus [\Fg^{-\th},x]$  
Thus by taking the $-\th$-part in 
(4.4.1), we have 
\begin{equation*}
\tag{4.4.2}
[\Fg^{\th},x] \oplus U^{-\th} = \Fg^{-\th}.
\end{equation*}
Let $D$ be the subspace of $V$ spanned by 
\begin{equation*}
\{ v_{i,j}, v'_{i,j'} 
   \mid 1 \le i \le m, \mu^{(1)}_i + 1 \le j \le \nu_i, 
           1 \le j' \le \mu^{(2)}_i\}.
\end{equation*}
Then $D$ is complementary to $E^x_Gv$ (see [SS, 2.1] for the notation). 
 The last space 
coincides with $E^x_Hv$ by [SS, Lemma 2.2], where $E^x_H = \Lie Z_H(x)$. 
Then one can check that $U^{-\th}\oplus D$ is complementary to 
the space $\{ ([z,x], zv) \mid z \in \Fg^{\th}\}$ in 
$\Fg^{-\th} \oplus V$, which is the tangent space at $(x,v)$ of 
the $H$-orbit $\CO_{\Bmu}$ 
containing $(x,v)$.  Put $S = (x,v) + U^{-\th} \oplus D$.  
Then $S$ is a transversal slice in $\Fg^{-\th} \oplus V$ 
for $\CO_{\Bmu}$ at $(x,v)$, in the following sense, 
$S \cap \CO_{\Bmu} = \{(x,v)\}$ and the tangent space of $\CO_{\Bmu}$
at $(x,v)$ is complementary to $S$. 
\par
Let $\xi': \Bk^* \to G$ be the one parameter subgroup defined by
\begin{equation*}
\xi'(t)v_{i,j} = t^{j-\mu^{(1)}_i-1}v_{i,j}, \quad
\xi'(t)v'_{i,j} = t^{\mu^{(2)}_i - j}v'_{i,j}. 
\end{equation*}
We define an action $\xi$ of $\Bk^*$ on $\Fg\oplus V$ by 
$\xi(t)(z , w) = (\Ad(\xi'(t))tz, \xi'(t)tw)$, where $tz, tw$
denote the scalar action of $\Bk^*$ on $\Fg$, $V$.
Then by definition, $\xi(t)$ fixes 
$v = \sum_{i=1}^{m}v_{i,\mu^{(1)}_i}$, and acts on $D$
linearly with strictly positive weights.
Since $xv_{i,j} = v_{i, j-1}$ or zero, and 
$xv'_{i,j} = v'_{i,j+1}$ or zero,  
$\xi(t)$ fixes $x$.  Hence $\xi(t)$ fixes $(x,v)$.  
Moreover, we have
\begin{equation*}
\xi(t)z_{i_1,i_2,s} = t^{\mu^{(1)}_{i_2} - \mu^{(1)}_{i_1} + s +1}
       z_{i_1,i_2,s},
\end{equation*}
and similar formulas also hold for 
$z_{i^*_1,i_2,s}, z_{i_1,i_2^*,s}, z_{i_1^*,i_2^*,s}$ with 
the same weight.  By the condition on $s$, we see that 
$\Bk^*$ acts on $U$ linearly with strictly positive weights. 
(In fact, this is clear if $i_2 \ge i_1$.  If $i_2 \le i_1$, 
we have $\mu_{i_2}^{(1)} - \mu_{i_1}^{(1)} + s 
  \ge \mu_{i_1}^{(2)} - \mu_{i_2}^{(2)} \ge 0$.) 
This implies that $\xi(t)$ stabilizes the subspace $U^{-\th}$.
Summing up the above argument, we see that 
$\Bk^*$ acts on the transversal slice $S$ through $\xi$, 
linearly with strictly positive weights with the origin $(x,v)$. 
Now the proposition follows by the argument due to Lusztig 
[L3, V, Proposition 24.6], who 
proved the purity of character sheaves in the case of 
reductive groups.
\end{proof}

\remark{4.5.}
Under the embedding 
$G^{\io\th}\uni \simeq G^{\io\th}\uni \times \{ 0\} \hra \CX\uni$, 
each $H$-orbit in $G^{\io\th}\uni$ is regarded as an $H$-orbit in 
$\CX\uni$, which gives an embedding 
$G^{\io\th}\uni/\!\sim_H \hra \CX/\!\sim_H$, 
$\CO_{\la} \mapsto \CO_{\Bla}$ with $\Bla = (-; \la)$. 
This embedding preserves the closure relations, hence the closure 
relations of $H$-orbits in $G^{\io\th}\uni$ are described from the
closure relations of $H$-orbits in $\CX\uni$, which is given by 
[AH, Theorem 6.3] (see [SS, (1.7.3)]).  Thus for 
$\CO_{\mu}, \CO_{\la} \in G^{\io\th}\uni$, one sees that 
$\CO_{\mu} \subset \ol\CO_{\la}$ if and only if $\mu \le \la$ 
with respect to the dominance order on $\CP_n$.   
Moreover the purity result in Proposition 4.4
holds also for $\IC(\ol\CO,\Ql)$ attached to an $H$-orbit  
$\CO \subset G^{\io\th}\uni$.

\par\bigskip

\section{Kostka polynomials}  

\para{5.1.}
In this section, we prove that the intersection cohomology of 
the closure of $H$-orbits in $\CX\uni$ can be interpreted 
in terms of Kostka polynomials introduced in [S2].
\par
First we review some results on Kostka polynomials.
Kostka functions associated to complex reflection groups 
were introduced in [S3],[S2] as a generalization of classical Kostka 
polynomials.  
In this section, we concentrate ourselves to a special case where 
Kostka functions are associated to ``limit symbols'' ([S2]), and 
related to the Weyl groups of either type $C_n$ or 
type $A_{n-1}$.  Apriori Kostka functions are rational functions.
However, in the latter case, they are just classical Kostka polynomials, 
and in the former case, it is proved in [S2, Proposition 3.3] that 
Kostka functions (and modified Kostka functions) are actually polynomials. 
In order to discuss both cases simultaneously, we introduce
some notations.  For $r = 1,2$, put 
$W_{n,r} = S_n \ltimes (\BZ/r\BZ)^n$. 
Hence $W_{n,r}$ is the Weyl group $W_n$ of type $C_n$ if $r = 2$, and 
the symmetric group $S_n$ if $r = 1$.  In the case where $r = 1$ 
Kostka polynomials $K_{\la,\mu}(t)$ associated to 
$\la, \mu \in \CP_n$ are nothing but the classical Kostka polynomials 
defined as the coefficients of 
the transition matrix between 
Schur functions $s_{\la}(x)$ and Hall-Littlewood functions 
$P_{\mu}(x;t)$.  Similarly, for the case $r = 2$, Kostka polynomials
$K_{\Bla, \Bmu}(t)$ associated to the double partitions 
$\Bla, \Bmu \in \CP_{n,2}$ are defined as the coefficients of 
the transition matrix between Schur functions $s_{\Bla}(x)$ and 
$P_{\Bmu}(x;t)$, where $P_{\Bmu}(x;t)$ is the Hall-Littlewood 
functions introduced in [S2].  We define a modified Kostka polynomial
$\wt K_{\Bla,\Bmu}(t)$ by 
\begin{equation*}
\wt K_{\Bla, \Bmu}(t) = t^{a(\Bmu)}K_{\Bla,\Bmu}(t\iv),
\end{equation*}
where the $a$-function $a(\Bmu)$ is defined, for 
$\Bmu = (\mu^{(1)}, \mu^{(2)})$, by 
\begin{equation*}
\tag{5.1.1}
a(\Bmu) = 2\cdot n(\Bmu) + |\mu^{(2)}|.
\end{equation*}
Here $n(\Bmu) = n(\mu^{(1)}) + n(\mu^{(2)})$ with the usual 
$n$-function $n(\mu) = \sum_{i=1}^k(i-1)\mu_i$ for 
a partition $\mu = (\mu_1, \dots, \mu_k)$.
Following [S2], we give a combinatorial characterization of 
Kostka polynomials $K_{\Bla,\Bmu}(t)$ and $K_{\la,\mu}(t)$.
For a (not necessarily irreducible) character $\x$ of $W_{n,r}$, 
we define the fake degree $R(\x)$ by 
\begin{equation*}
\tag{5.1.2}
R(\x) = \frac{\prod_{i=1}^n(t^{ir} - 1)}{|W_{n,r}|}
            \sum_{w \in W_{n,r}}\frac{\ve(w)\x(w)}{\det_{V_0}(t - w)},
\end{equation*}
where $\ve$ is the sign character of $W_{n,r}$, and 
$\det_{V_0}(t-w)$ means the determinant of $t-w$ on the $W_{n,r}$-module 
$V_0$; here $V_0$ is the reflection representation of $W_n$ for 
$r = 2$, and its restriction on $S_n$ if $r = 1$.
Note that $R(\x) \in \ZZ_{\ge 0}[t]$; if $\x$ is irreducible, 
$R(\x)$ coincides with the graded multiplicity of $\x$ in the
coinvariant algebra $R(W_{n,r})$  of $W_{n,r}$.
We define a square matrix 
$\Om = (\w_{\Bla,\Bmu})_{\Bla, \Bmu \in \CP_{n,r}}$ by 
\begin{equation*}
\tag{5.1.3}
\w_{\Bla,\Bmu} = t^NR(\x^{\Bla}\otimes \x^{\Bmu}\otimes\ve),
\end{equation*} 
where $N$ is the number of reflections of $W_{n,r}$, which 
is also given as the maximal degree of $R(W_{n,r})$, and $\x^{\Bla}$
is the irreducible character of $W_{n,r}$ corresponding to 
$\Bla \in \CP_{n,r}$.
Here $\w_{\Bla,\Bmu}$ is a polynomial in $t$, which we denote by 
$\w_{\Bla,\Bmu}(t)$.  In the case where $r = 2$, 
it is known that $\det(q - w) = |T^{\th, F}_w|$ for each 
$w \in W_n$, where $T_w^{\th}$ is an $F$-stable 
maximal torus of $H$ corresponding to $w \in W_n$.  Then (5.1.3) can be 
written as 
\begin{equation*}
\tag{5.1.4}
\w_{\Bla,\Bmu}(q) = |H^F||W_n|\iv\sum_{w \in W_n}
    |T_w^{\th, F}|\iv \x^{\Bla}(w)\x^{\Bmu}(w).
\end{equation*}   
In the case where $r= 1$, a similar formula holds as above;
\begin{equation*}
\tag{5.1.5}
\w_{\la,\mu}(q) = |GL_n^F||S_n|\iv\sum_{w \in S_n}|S_w^F|\iv
                   \x^{\la}(w)\x^{\mu}(w),
\end{equation*}
where $S_w$ is an $F$-stable maximal torus of $GL_n$ corresponding to 
$w \in S_n$.
\par
Recall the partial order $\Bmu \le \Bla$ on $\CP_{n,2}$ defined 
in [SS, 1.7].  We have the following result.

\begin{thm}[{[S2, Theorem 5.4]}]  
Assume that $r = 2$. 
There exist unique matrices $P = (p_{\Bla,\Bmu}), 
\vL = (\xi_{\Bla,\Bmu})$ over $\BQ[t]$ satisfying the equation
\begin{equation*}
P\vL\,{}^t\!P = \Om
\end{equation*} 
subject to the condition that $\vL$ is a diagonal matrix and that
\begin{equation*}
p_{\Bla, \Bmu} = \begin{cases}
           0 &\quad\text{ unless } \Bmu \le \Bla,  \\
           t^{a(\Bla)}  &\quad\text{ if } \Bla = \Bmu.
                 \end{cases}
\end{equation*}
Then the entry $p_{\Bla,\Bmu}$ of the matrix $P$ coincides with 
$\wt K_{\Bla, \Bmu}(t)$.
\end{thm}

\remark{5.3.}
In the case where $r = 1$, it is known that the modified Kostka
polynomials $\wt K_{\la,\mu}(t)$ 
are characterized by a similar formula;  consider 
a matrix equation $P\vL\, {}^t\!P = \Om$ as in the theorem, where 
the partial order on $\CP_{n,2}$ is replaced by the dominance order 
$\mu \le \la$ on $\CP_n$, and the function $a(\Bla)$ is replaced by 
$n(\la)$. Then each entry $p_{\la,\mu}$ coincides with 
$\wt K_{\la,\mu}(t)$.   

\para{5.4.}
Let $\ZC_q(\CX\uni)$ 
(resp. $\ZC_q(G^{\io\th}\uni)$) be the $\Ql$-space of 
all $H^F$-invariant $\Ql$-functions  on $\CX\uni^F$ 
(resp. on $(G^{\io\th}\uni)^F$).  
We define a bilinear form $\lp f, h\rp_{\ex}$ on $\ZC_q(\CX\uni)$ by 
\begin{equation*}
\lp f, h\rp_{\ex} = \sum_{z \in \CX\uni^F}f(z)h(z),
\end{equation*} 
and similarly define a bilinear form $\lp f, h\rp_{\sym}$ on 
$\ZC_q(G^{\io\th}\uni)$
by replacing $\CX\uni^F$ by $(G^{\io\th}\uni)^F$ in the above formula.
For each $\Bla \in \CP_{n,2}$, we define 
$Q_{\Bla} \in \ZC_q(\CX\uni)$ by 
\begin{equation*}
\tag{5.4.1}
Q_{\Bla} = |W_n|\iv \sum_{w \in W_n}\x^{\Bla}(w)Q_{T_w}.
\end{equation*}
In view of (4.1.2), we have
$Q_{\Bla} = (-1)^{\dim \CX - \dim \CX\uni}\x_{A_{\Bla}, \vf_{\Bla}}$, 
where $\vf_{\Bla} : F^*A_{\Bla} \isom A_{\Bla}$ is as in 4.1.
Next, under the notation in 4.2, we define 
$Q_{\la}^{\sym} \in \ZC_q(G^{\io\th}\uni)$ by 
\begin{equation*}
Q^{\sym}_{\la} = |S_n|\iv\sum_{w \in S_n}\x^{\la}(w)
             \Psi_w(q)\iv Q_{T_w}^{\sym}.
\end{equation*} 
By (4.2.2), $Q_{\la}^{\sym} = 
(-1)^{\dim G^{\io\th} - \dim G^{\io\th}\uni}\x_{A_{\la}, \vf_{\la}}$, 
where $\vf_{\la}: F^*A_{\la} \isom A_{\la}$ is as in 4.2.
We have the following result.

\begin{prop}   
\begin{enumerate}
\item 
For each $\Bla, \Bmu \in \CP_{n,2}$, we have 
\begin{equation*}
\lp Q_{\Bla}, Q_{\Bmu}\rp_{\ex} = 
       |H^F||W_n|\iv\sum_{w \in W_n}|T_w^{\th, F}|\iv
                                    \x^{\Bla}(w)\x^{\Bmu}(w). 
\end{equation*} 
\item 
For each $\la, \mu \in \CP_n$, we have 
\begin{equation*}
\lp Q^{\sym}_{\la}, Q_{\mu}^{\sym} \rp_{\sym}
      = |GL_n^{F^2}||S_n|\iv
      \sum_{w \in S_n}|T_w^{\io\th, F^2}|\iv \x^{\la}(w)\x^{\mu}(w).
\end{equation*}
\end{enumerate}
\end{prop} 

\begin{proof}
We show (i).  By the orthogonality relations for Green functions
(Theorem 3.5), we have 
\begin{equation*}
\begin{split}
\lp Q_{\Bla}, Q_{\Bmu}\rp_{\ex} &= |W_n|^{-2}\sum_{w, w' \in W_n}
         \x^{\Bla}(w)\x^{\Bmu}(w') \lp Q_{T_w}, Q_{T_{w'}}\rp_{\ex} \\
   &= |H^F||W_n|^{-2}\sum_{w,w' \in W_n}
 \x^{\Bla}(w)\x^{\Bmu}(w')\d_{w,w'}|N_H(T^{\th}_w)^F||T_w^{\th, F}|^{-2},
\end{split}
\end{equation*}
where $\d_{w,w'} = 1$ if $w$ and $w'$ are conjugate in $W_n$ and 
$\d_{w,w'} = 0$ otherwise.
Since $|N_H(T_w^{\th})^F|/|T_w^{\th, F}| = 
|(N_H(T_0^{\th})/T_0^{\th})^{wF}| = |Z_{W_n}(w)|$, 
we obtain the formula in (i). 
\par
We show (ii).  By  the orthogonality relations for Green functions
(Theorem 3.12), a similar computation as above implies that
\begin{equation*}
\begin{split}
\lp Q^{\sym}_{\la}, Q_{\mu}^{\sym}\rp_{\sym}
  = |H^F|&|S_n|^{-2}\sum_{w,w' \in S_n}\x^{\la}(w)\x^{\mu}(w') \times \\
               & \times \d_{w,w'}q^{-2n}\Psi_w(q)^{-2}
       |N_H(T_w^{\io\th})^F||T_w^{\io\th F}|^{-2},
\end{split}
\end{equation*} 
where $\d_{w,w'} = 1$ if $w, w'$ are conjugate under $S_n$, and 
$\d_{w,w'} = 0$ otherwise.
Now 
\begin{equation*}
|N_H(T_w^{\io\th})^F|/|Z_H(T_w^{\io\th})^F| = 
|(N_H(T^{\io\th}_0)/Z_H(T_0^{\io\th}))^{wF}| = |Z_{S_n}(w)|
\end{equation*}
since $N_H(T_0^{\io\th})/Z_H(T_0^{\io\th}) = \CW \simeq S_n$, on which 
$F$ acts trivially.  Moreover we have 
\begin{equation*}
Z_H(T_w^{\io\th})^F \simeq Z_H(T_0^{\io\th})^{wF} \simeq ((SL_2)^n)^{wF},
\end{equation*}
where $w$ acts on $(SL_2)^n$ by a permutation of factors.
Assume that $w$ is of type $\nu = (\nu_1, \dots, \nu_k)$.  Then 
\begin{equation*}
|Z_H(T_w^{\io\th})^F| = q^n\prod_{i=1}^k(q^{2\nu_i}-1).
\end{equation*}
Since $|T_w^{\io\th, F}| = \prod_{i=1}^k(q^{\nu_i}-1)$, we have
\begin{equation*}
\Psi_w(q)^{-2}|T_w^{\io\th, F}|^{-2}|Z_H(T_w^{\io\th})^F| = 
    q^n(\prod_{i=1}^k(q^{2\nu_i}-1))\iv = q^n|T_w^{\io\th, F^2}|\iv.
\end{equation*}
By noticing that 
$q^{-n}|H^F| = q^{(n^2 -n)}\prod_{i=1}^n(q^{2i} - 1) = |GL_n^{F^2}|$, 
we obtain the formula in (ii).
\end{proof}

\para{5.6.}
For $\Bla, \Bmu \in \CP_{n,2}$, we define a polynomial
$\IC_{\Bla, \Bmu}(t) \in \BZ[t]$ by 
\begin{equation*}
\tag{5.6.1}
\IC_{\Bla,\Bmu}(t) = 
  \sum_{i \ge 0}\dim \CH^{2i}_z\IC(\ol\CO_{\Bla},\Ql)t^i,
\end{equation*}
where $\CO_{\Bla}$ is an $H$-orbit in $\CX\uni$, and 
$z \in \CO_{\Bmu} \subset \ol\CO_{\Bla}$.  
Similarly for $\la, \mu \in \CP_n$, we define a polynomial 
$\IC^{\sym}_{\la,\mu}(t)$ by 
\begin{equation*}
\tag{5.6.2}
\IC_{\la,\mu}^{\sym}(t) = 
   \sum_{i \ge 0}\dim \CH^{2i}_x\IC(\ol\CO_{\la},\Ql)t^i,
\end{equation*}
where $\CO_{\la}$ is an $H$-orbit in $G^{\io\th}\uni$ and 
$x \in \CO_{\mu} \subset \ol\CO_{\la}$.
Now we prove the following theorem, which was conjectured
by Achar-Henderson  in [AH, Conjecture 6.4].

\begin{thm}  
Let $\Bla, \Bmu \in \CP_{n,2}$.  
\begin{enumerate}
\item
$\IC_{\Bla,\Bmu}(t) = t^{-a(\Bla)}\wt K_{\Bla,\Bmu}(t)$.
\item
$\CH^i\IC(\ol\CO_{\Bla},\Ql) = 0$ unless $i \equiv 0 \pmod 4$.
 \item
 The eigenvalues of 
$\f_{\Bla}$ on $\CH^{4i}_z\IC(\ol\CO_{\Bla},\Ql)$ are 
$q^{2i}$ for any $z \in \ol\CO_{\Bla}$.
  \end{enumerate}
 \end{thm}

 \begin{proof}
 Let $Y_{\Bmu} \in \ZC_q(\CX\uni)$ 
be the characteristic function on $\CO_{\Bmu}^F$
for each $\Bmu \in \CP_{n,2}$, i.e., $Y_{\Bmu}(z) = 1$ if 
$z \in \CO_{\Bmu}^F$ and $Y_{\Bmu}(z) = 0$ otherwise.  Then 
$\{ Y_{\Bmu} \mid \Bmu \in \CP_{n,2}\}$ gives a basis of 
$\ZC_q(\CX\uni)$.  Since 
$A_{\Bla}$ is an $H$-equivariant perverse sheaf, we have 
$\x_{A_{\Bla}, \vf_{\Bla}} \in \ZC_q(\CX\uni)$.  
Hence $\x_{A_{\Bla}, \vf_{\Bla}}$
is written as 
$\x_{A_{\Bla},\vf_{\Bla}} = \sum_{\Bmu}p_{\Bla,\Bmu}Y_{\Bmu}$,
where $Y_{\Bmu}$ appearing in the sum satisfies the condition that 
$\CO_{\Bmu} \subset \ol\CO_{\Bla}$, which is by [AH, Theorem 6.3] 
(see [SS, (1.7.3)]) equivalent to $\Bmu \le \Bla$.  
Here we note that 
\begin{equation*}
\tag{5.7.1}
d_{\Bla} = a(\Bla)
\end{equation*}
by [SS, Lemma 2.3].  Then together with (4.1.1), 
we see that $p_{\Bla,\Bla} = q^{d_{\Bla}} = q^{a(\Bla)}$.
We consider the matrix $P = (p_{\Bla,\Bmu})$, and  
define a matrix $\vL = (\xi_{\Bla, \Bmu})$ by 
$\xi_{\Bla,\Bmu} = \lp Y_{\Bla}, Y_{\Bmu}\rp_{\ex}$.  Then 
$\vL$ is a diagonal matrix. 
Now it follows from the equation in Proposition 5.5 (i), 
together with the formula (5.1.4) that the matrices $P,\vL, \Om$ 
satisfies the relation $P\vL\, {}^t\!P = \Om(q)$, where 
$\Om(q)$ denotes the matrix $\Om$ evaluated by $t = q$ for each
entry $\w_{\Bla,\Bmu}(t)$.  
We write $p_{\Bla,\Bmu}$ as $p_{\Bla,\Bmu}(q)$ to indicate its 
dependence on $\Fq$.  Then by Theorem 5.2, we see that 
$p_{\Bla,\Bmu}(q) = \wt K_{\Bla,\Bmu}(q)$.   On the other hand, 
one can write 
\begin{equation*}
\tag{5.7.2}
p_{\Bla,\Bmu}(q) = \x_{A_{\Bla},\vf_{\Bla}}(z) 
= \sum_{i}(-1)^iTr(\vf^*_{\Bla}, \CH^i_zA_{\Bla})
\end{equation*}
for $z \in \CO_{\Bmu}^F$, where $\vf_{\Bla}^*$ is the map 
on $\CH^i_zA_{\Bla}$ induced from $\vf_{\Bla}$.  
If we replace $\Fq$ by its extension 
$\BF_{q^m}$ for a positive integer $m$, the above formula is 
replaced by 
\begin{equation*}
\tag{5.7.3}
p_{\Bla,\Bmu}(q^m) = \sum_{i}(-1)^iTr((\vf^*_{\Bla})^m, \CH^i_zA_{\Bla})
\end{equation*}
for a fixed $z \in \CO_{\Bmu}^F$.
Since $\dim \CO_{\Bla}$ is even, $p_{\Bla,\Bmu}(q)$ coincides with 
the formula obtained from (5.7.2) by replacing $A_{\Bla}$ by  
$\IC(\ol\CO_{\Bla},\Ql)$.
Thus also in (5.7.3), we may replace $A_{\Bla}$ by 
$\IC(\ol\CO_{\Bla}, \Ql)$.   By (4.1.1) and by Proposition 4.4, 
$q^{-d_{\Bla}}\vf_{\Bla}$ is pointwise pure.  This implies that 
$p_{\Bla,\Bmu}(q^m)$ can be written as 
\begin{equation*}
\tag{5.7.4}
p_{\Bla,\Bmu}(q^m) = 
  q^{md_{\Bla}}\sum_i(-1)^i\sum_{j = 1}^{k_i}(\a_{i,j}q^{i/2})^m,
\end{equation*}
where $k_i = \dim \CH_z^i$ for $\CH_z^i = \CH^i_z\IC(\ol\CO_{\Bla},\Ql)$, and 
$\{ \a_{i,j}q^{i/2} \mid 1 \le j \le k_i\}$ are eigenvalues of 
$\vf^*_{\Bla}$ on $\CH_z^i$ such that $\a_{i,j}$ are algebraic integers
all of whose complex conjugate have absolute value 1.
Since $p_{\Bla,\Bmu}(q^m) = \wt K_{\Bla, \Bmu}(q^m)$, and 
$\wt K_{\Bla, \Bmu}(t)$ is a polynomial in $t$, by Dedekind's theorem, 
(5.7.4) implies that $\a_{i,j} = 0$ for odd $i$, and 
$\a_{i,j} = 1$ for even $i$ if $\CH_z^{i} \ne 0$.   
This implies that $\CH_z^i = 0$  
for odd $i$ and 
$\sum_{j=1}^{k_i}\a_{i,j}q^{i/2} = (\dim \CH_z^{i})q^{i/2}$
for even $i$. 
Hence by (5.7.1) we see that 
$\IC_{\Bla,\Bmu}(t) = t^{-a(\Bla)}\wt K_{\Bla,\Bmu}(t)$ and 
(i) holds.  
\par
Next we show (ii).  It is enough to see that
\begin{equation*}
\tag{5.7.5} 
t^{-a(\Bla)}\wt K_{\Bla,\Bmu}(t) \in \BZ[t^2].
\end{equation*} 
Let $D$ be the diagonal matrix whose $(\Bla,\Bla)$-entry is
$t^{-a(\Bla)}$.
Then by Theorem 5.2, $t^{-a(\Bla)}\wt K_{\Bla,\Bmu}(t)$ 
is given as the $(\Bla,\Bmu)$-entry of the matrix $DP$.
By Theorem 5.2, we have a matrix equation 
$(DP)\vL\, {}^t(DP) = D\Om D$.  Here $DP$ is a lower 
unitriangular matrix if we consider the matrix with respect to a
total order compatible with the partial order $\le$ on $\CP_{n,2}$. 
Hence $DP$ is determined uniquely from 
this matrix equation.  The $(\Bla, \Bmu)$-entry of 
$D\Om D$ is given by $t^{-a(\Bla)-a(\Bmu)}\w_{\Bla,\Bmu}(t)$.
We note that 
\begin{equation*}
\tag{5.7.6}
t^{-a(\Bla)-a(\Bmu)}\w_{\Bla,\Bmu}(t) = 
(-t)^{-a(\Bla)-a(\Bmu)}\w_{\Bla,\Bmu}(-t).
\end{equation*}
In fact, $W_n$ contains $-1$ as a central element, and 
$\x^{\Bla}(-w) = \x^{\Bla}(-1)\x^{\Bla}(w)$. 
From the construction of $\x^{\Bla}$, we see 
that $\x^{\Bla}(-1) = (-1)^{|\la^{(2)}|}$. 
Thus by (5.1.2) and (5.1.4), we have 
$\w_{\Bla,\Bmu}(-t) = 
(-1)^{|\la^{(2)}| + |\mu^{(2)}|}\w_{\Bla,\Bmu}(t)$. 
By (5.1.1), we have $(-t)^{a(\Bla)} = (-1)^{|\la^{(2)}|}t^{a(\Bla)}$.
Hence (5.7.6) follows.   
Now by the uniqueness of $DP$ in the above matrix equation, we see
that each entry of $DP$ is not changed by replacing $t$ by $-t$.
Hence (5.7.5) holds, and (ii) is proved.
\par
(iii) follows from the discussion in (i) and (ii).
The theorem is proved.
\end{proof}

\remarks{5.8}  (i) In [AH], Achar-Henderson suggested a way for the
proof of their conjecture.  This idea was recently carried out 
by Kato (in the case where $\Bk = \BC$). 
He showed in [Ka3, Theorem A], for each $z \in \CO_{\Bmu}$, 
 that the cohomology ring $H^{\bullet}(\CB_z, \BC)$ has 
a De Concini-Procesi type interpretation as in the case of 
$GL_n$ ([DP]), i.e., there exists a graded algebra isomorphism 
between $H^{\bullet}(\CB_z,\BC)$ and 
$R_{\Bmu} = \BC[x_1, \dots, x_n]/I_{\Bmu}$, 
compatible with the action of $W_n$, 
where $I_{\Bmu}$ is the ideal of all polynomials 
$p(x_1, \dots, x_n)$ such that 
$p(\partial/\partial x_1, \dots, \partial/\partial x_n)$
annihilates the Specht module $S_{\Bmu}$ realized in the 
homogeneous component of $\BC[x_1, \dots, x_n]$ of degree $a(\Bmu)$. 
Let $R^i_{\Bmu}$ be the $W_n$-module obtained as the $i$-th homogeneous 
part of $R_{\Bmu}$.  It was conjectured in [S2, 3.13] that 
for any irreducible character $\x^{\Bla}$ of $W_n$, we have  
\begin{equation*}
\sum_{i \ge 0}\lp R^i_{\Bmu}, \x^{\Bla}\rp_{W_n}t^i = \wt K_{\Bla,\Bmu}(t).
\end{equation*}
Kato proved this conjecture in [Ka5, Theorem A.2] and in
[Ka4, Theorem A.1~8].  Thus combined with his earlier result
in [Ka3], Kato's results provide a proof of Achar-Henderson's 
conjecture. 
\par
Our approach for the conjecture is based on 
another idea explained in [AH]. 
\par 
(ii) \ Let $G = GL(V)$ with $\dim V = n$, and $\Fg\nil$ be 
the nilpotent cone of $\Fg = \Lie G$.  
The $G$-orbits of $\Fg\nil$ are parametrized by $\CP_n$. 
The polynomial $\IC_{\la,\mu}(t)$ is defined similar 
to  (5.6.1) or (5.6.2), by using the intersection cohomology 
associated to the nilpotent orbit corresponding to $\la \in \CP_n$.
It is a well-known result by Lusztig [L1] that
\begin{equation*}
\tag{5.8.1}
\IC_{\la,\mu}(t) = t^{-n(\la)}\wt K_{\la,\mu}(t),
\end{equation*} 
where $\wt K_{\la,\mu}(t)$ is the (modified) 
classical Kostka polynomial 
indexed by $\la, \mu \in \CP_n$.
We now consider the variety $\Fg\nil \times V$, which is called 
the enhanced nilpotent cone by [AH].  It is known that $G$-orbits of 
$\Fg\nil \times V$ are parametrized naturally by $\CP_{n,2}$.  
The polynomial $\IC^{\en}_{\Bla, \Bmu}(t)$ is defined similarly 
by using the intersection cohomology associated the the orbit in 
$\Fg\nil \times V$ corresponding to $\Bla \in \CP_{n,2}$. 
 In [AH, Theorem 5.2], 
Achar-Henderson proved, as an analogy of (5.8.1),  that 
\begin{equation*}
\tag{5.8.2}
\IC^{\en}_{\Bla,\Bmu}(t^2) = t^{-a(\Bla)}\wt K_{\Bla,\Bmu}(t).
\end{equation*}
The fact that $\wt K_{\Bla,\Bmu}(t)$ is a polynomial in $t$ 
and the property (5.7.5) follow also from this. 

\para{5.9.}
Next we show a similar formula in the case of $G^{\io\th}\uni$.
The following theorem was first proved by Henderson in 
[H1, Theorem 6.3].  However his argument depends in part on 
a result from [BKS].  Our argument is independent from [BKS], 
and it is proved in a similar way as the case of $\CX\uni$.

\begin{thm}  
Assume that $\la, \mu \in \CP_n$.
\begin{enumerate}
\item
$\IC^{\sym}_{\la,\mu}(t) = t^{-2n(\la)}\wt K_{\la,\mu}(t^2)$.
\item
$\CH^i\IC(\ol\CO_{\la},\Ql) = 0$ unless $i \equiv 0 \pmod 4$.
\item
The eigenvalues of $\f_{\la}$ on $\CH^{4i}_x\IC(\ol\CO_{\la},\Ql)$
are $q^{2i}$ for any $x \in \ol\CO_{\la}$. 
\end{enumerate}
\end{thm}  

\begin{proof}
Let $Y_{\mu} \in \ZC_q(G^{\io\th}\uni)$ be the characteristic function 
of $\CO_{\mu}^F$ for each $\mu \in \CP_n$.   Then 
$\{ Y_{\mu} \mid \mu \in \CP_n \}$ gives a basis of 
$\ZC_q(G^{\io\th}\uni)$. 
Since $\x_{A_{\la},\vf_{\la}} \in \ZC_q(G^{\io\th}\uni)$, one can write
$\x_{A_{\la},\vf_{\la}} = \sum_{\mu \in \CP_n}p_{\la,\mu}Y_{\mu}$
with $p_{\la,\mu} \in \Ql$, where $Y_{\mu}$ appearing in the sum
satisfies the condition $\CO_{\mu} \subset \ol\CO_{\la}$, which is
equivalent to the condition that $\mu \le \la$ 
with respect to the dominance order on $\CP_n$ (see Remark 4.5).
Here we note that
\begin{equation*}
\tag{5.10.1}
d_{\la} = 2n(\la).
\end{equation*}
Hence by (4.2.1), we have $p_{\la,\la} = q^{d_{\la}} = q^{2n(\la)}$.
Let $P = (p_{\la,\mu})$ and $\vL = (\xi_{\la,\mu})$, where 
$\xi_{\la,\mu} = \lp Y_{\la}, Y_{\mu}\rp_{\sym}$. 
Let $\Om(t) = (\w_{\la,\mu}(t))$ be the matrix defined in 5.1.
Then by comparing Proposition 5.5 (ii) with (5.1.5), 
we have a matrix relation 
$P\vL\,{}^t\!P = \Om(q^2)$.  
(Note that $T_w^{\io\th, F} \simeq S_w^F$ for $w \in S_n$.) 
We write $p_{\la,\mu}$ as 
$p_{\la,\mu}(q)$ 
to indicate the $\Fq$-structure.  Then by Remark 5.3, we see that
$p_{\la,\mu}(q) = \wt K_{\la,\mu}(q^2)$.
Now by a similar argument as in the proof of Theorem 5.7 
(the purity result such as Proposition 4.4 holds in this case, 
see Remark 4.5), 
we see that $\IC^{\sym}_{\la,\mu}(t) = t^{-2n(\la)}\wt K_{\la,\mu}(t^2)$, 
which proves (i).  The statements (ii) and (iii) follows from this.
\end{proof}

\par\bigskip

\section{ $H^F$-invariant functions on $\CX^F$ }

\para{6.1.}
We consider the complex $K_{T,\CE}$ as in (1.3.1). 
Then by [SS, 3.5], $K_{T,\CE}$ can be written in the form
\begin{equation*}
\tag{6.1.1}
K_{T,\CE} \simeq \bigoplus_{\r \in \wt\CW_{\CE}\wg}\r \otimes K_{\r},
\end{equation*}
where $\CW_{\CE}$ is the stabilizer of $\CE$ in $\CW$ and 
$\wt\CW_{\CE} = \CW_{\CE}\ltimes (\BZ/2\BZ)^n$,   
$K_{\r}$ is a simple perverse sheaf of the form  
$\IC(\CX_m, \CL)[d_m]$ for a simple local system $\CL$ on 
$\CY_m^0$.  
Assume that $T = T_0$, and for a tame local system $\CE$ on 
$T_0^{\io\th}$, let 
$\CY_{\CE} = \{ w \in \CW \mid  (wF)^*\CE \simeq \CE \}$. 
We assume that $\CY_{\CE} \ne \emptyset$. 
Then there exists $w_1 \in \CW$ such that $\CY_{\CE} = \CW_{\CE}w_1$.
We put $\wt\CY_{\CE} = \wt\CW_{\CE}w_1$. 
Here $\CW_{\CE}$ is a Weyl subgroup of $\CW \simeq S_n$, and we may choose
$w_1$ so that $\g = w_1F$ acts on $\CW_{\CE}$ as a permutation of factors.
We consider the semidirect product $\wt\CW_{\CE}\ltimes \lp\g\rp$, where
$\lp \g\rp$ is an infinite cyclic group generated by $\g$. 
Then  the action of $\g$ on $\CW_{\CE}$ is extended to $\wt\CW_{\CE}$. 
We have a canonical isomorphism 
$\vf : (F_1)^*K_{T_0,\CE} \isom K_{T_0,\CE}$ by a similar argument 
as in 1.3.
If $\r \in \wt\CW_{\CE}\wg$ is $w_1F$-stable, the direct 
summand $K_{\r}$ in (6.1.1) is $w_1F$-stable.  
We denote by $(\wt\CW_{\CE})\wg_{\ex}$ the set of $w_1F$-stable 
characters of $\wt\CW_{\CE}$, and denote by $\vf_1$ 
the restriction of $\vf$ on the whole sum of $\r\otimes K_{\r}$ 
for $\r \in (\wt\CW_{\CE})\wg_{\ex}$.
For each representation $\r \in (\wt\CW_{\CE})\wg_{\ex}$, 
we fix an extension 
$\wt\r$ of $\r$ to $\wt\CW_{\CE}\ltimes\lp\g\rp$, 
where the action of $\g$ is given by $\g_{\r}$.
 We define an isomorphism 
$\vf_{\r}: (F_1)^*K_{\r} \isom K_{\r}$ for each 
$\r \in (\wt\CW_{\CE})\wg_{\ex}$ by the condition 
that $\vf_1 = \sum_{\r \in (\wt\CW_{\CE})\wg_{\ex}}\g_{\r} \otimes \vf_{\r}$. 
We now consider a pair $(T_w, \CE_w)$ for $w \in \wt\CW \simeq W_n$, 
where $T_w = hTh\iv$ and 
$\CE_w = (h\iv)^*\CE$ as in 1.2. 
Note that the condition $\CE_w$ is an $F$-stable 
local system on $T^{\io\th}_w$ is 
equivalent to the condition that $(\bar wF)^*\CE \simeq \CE$, and 
so to the condition that $w \in \wt\CY_{\CE}$, where $\bar w$ is the 
image of $w$ under the map $\wt\CW \to \CW$.
We have a canonical isomorphism 
$\vf_w: F^*K_{T_w,\CE_w} \isom K_{T_w,\CE_w}$, and    
denote by $\x_{T_w, \CE_w}$ the corresponding function. 
Then by (1.7.3), the following 
formula holds for $w \in \wt\CW_{\CE}$.
\begin{equation*}
\tag{6.1.2}
\x_{T_{ww_1},\CE_{ww_1}} = 
     \sum_{\r \in (\wt\CW_{\CE})\wg_{\ex}}\x_{\wt\r}(w\g)
      \x_{K_{\r}, \vf_{\r}},
\end{equation*}
where $\x_{\wt\r}$ is the character of $\wt\r$. 
We show the following result.

\begin{prop}  
Let $\CE, \CE'$ be tame local systems on  
$T_0^{\io\th}$ such that 
$\CY_{\CE}\ne \emptyset, \CY_{\CE'} \ne \emptyset$. 
Then for $\r \in (\wt\CW_{\CE})\wg_{\ex}$,
$\r' \in (\wt\CW_{\CE'})\wg_{\ex}$, we have
\begin{equation*}
\tag{6.2.1}
\begin{split}
|H^F|\iv &\sum_{z \in \CX^F}
     \x_{K_{\r},\vf_{\r}}(z)\x_{K_{\r'},\vf_{\r'}}(z)  \\ 
 &= |\wt\CW_{\CE}|\iv \sum_{w \in \wt\CW_{\CE}}
       |T_{ww_1}^{\io\th, F}||T_{ww_1}^{\th, F}|\iv 
         \x_{\wt\r}(w\g)\x_{\wt\r'}(w\g)   
\end{split}
\end{equation*}
if $\CE'$ is isomorphic to the dual local system $\CE^{\vee}$ of 
$\CE$, and is equal to zero otherwise.
(Note that if $\CE' \simeq \CE\wg$, then 
$\wt\CW_{\CE'} \simeq \wt\CW_{\CE}$, and we regard 
$\r'$ as an element in $(\wt\CW_{\CE})\wg_{ex}$.)
\end{prop}

\begin{proof}
By (6.1.2), and by the orthogonality relations for extended 
irreducible characters (see [L3, (10.3.1)]), we have
\begin{equation*}
\x_{K_{\r},\vf_{\r}} = |\wt\CW_{\CE}|\iv \sum_{w \in \wt\CW_{\CE}}
       \x_{\wt\r}(w\g)\x_{T_{ww_1},\CE_{ww_1}}
\end{equation*}
and a similar formula holds for $\x_{K_{\r'}, \vf_{\r'}}$.
Let $\vth_w$ (resp. $\vth'_{w'}$) be the character of 
$T_{ww_1}^{\io\th, F}$ (resp. $T_{w'w_1'}^{\io\th, F}$) corresponding to
$\CE_w$ (resp. $\CE'_{w'}$).
Then by Theorem 3.4, the left hand side of  (6.2.1) is equal to
\begin{equation*}
\tag{6.2.2}
\begin{split}
|\wt\CW_{\CE}|\iv |\wt\CW_{\CE'}|\iv&\sum_{\substack{w \in \wt\CW_{\CE} \\
    w' \in \wt\CW_{\CE'}}}
       |T_{ww_1}^{\th, F}|\iv |T_{w'w_1'}^{\th, F}|\iv
      \x_{\wt\r}(w\g)\x_{\wt\r'}(w'\g')  \\
     &\times \sum_{\substack{n \in N_H(T_{ww_1}^{\th}, T_{w'w_1'}^{\th})^F \\
              t \in T_{ww_1}^{\io\th, F} }}\vth_w(t)\vth'_{w'}(n\iv tn).
\end{split}
\end{equation*}
Let $A_{w,w'} = \sum\vth_w(t)\vth'_{w'}(n\iv tn)$ be the second sum in (6.2.2).
Then $A_{w,w'} = 0$ unless $T_{ww_1}$ and $T_{w'w_1'}$ are conjugate by 
$n \in N_H(T^{\th}_{ww_1}, T^{\th}_{w'w_1'})^F$ such that 
$\vth'_{w'} = (\ad n)^*(\vth_w\iv)$, 
in which case 
\begin{align*}
A_{w,w'} &= \sharp\{ n \in N_H(T_{ww_1}^{\th}, T_{w'w_1'}^{\th})^F
            \mid \vth'_{w'} = (\ad n)^*(\vth_w\iv) \}|T_{ww_1}^{\io\th, F}| \\ 
    &= |Z_{\wt\CW_{\CE}}(w\s_{\r})||T_{ww_1}^{\th, F}||T_{ww_1}^{\io\th, F}| 
\end{align*}
since $N_H(T_{ww_1}^{\th}, T_{w'w_1'}^{\th})^F/T_{ww_1}^{\th, F} 
         \simeq Z_{\wt\CW}(w\s_{\r})$, 
and the condition $\vth'_{w'} = (\ad n)^*(\vth_w\iv)$
corresponds to the condition that the image 
$\bar n$ of $n$ to $\wt\CW$ gives an isomorphism 
$\ad (\bar n\iv): \wt\CW_{\CE} \isom \wt\CW_{\CE'}$, 
$\wt\CY_{\CE} \isom \wt\CY_{\CE'}$. 
Under this identification, $w_1$ is mapped to $w_1'$, and 
$w'\g'$ corresponds to an element  
conjugate to $w\g$ by an element in $\wt\CW_{\CE}$. 
Thus we may regard $\r'$ as an element in $(\wt\CW_{\CE})\wg_{\ex}$.
It follows that 
the sum in (6.2.1) turns out to be
\begin{equation*}
\begin{split}
&|\wt\CW_{\CE}|^{-2}\sum_{\substack{ w, w' \in \wt\CW_{\CE} 
         \\  w\g \sim w'\g' }}
|Z_{\wt\CW_{\CE}}(w\g)||T_{ww_1}^{\io\th, F}||T_{ww_1}^{\th, F}|\iv 
            \x_{\wt\r}(w\g)\x_{\wt\r'}(w'\g')  \\
&= |\wt\CW_{\CE}|\iv \sum_{w \in \wt\CW_{\CE}}
       |T_{ww_1}^{\io\th, F}||T_{ww_1}^{\th, F}|\iv 
          \x_{\wt\r}(w\g)\x_{\wt\r'}(w\g).
\end{split}
\end{equation*}    
Since the condition $\vth'_{w'} = (\ad n)^*\vth_w\iv$ is 
equivalent to $\CE' \simeq \CE^{\vee}$, the proposition 
follows. 
\end{proof}

\para{6.3.}
In analogy to the case of $GL_n(\BF_q)$, we shall give 
a parametrization of $H^F$-orbits in $\CX^F$.  Recall that 
$\CP_{n,2}$ is the set of double partitions of $n$.  
Let
$\CP^{(2)} = \coprod_{n = 0}^{\infty}\CP_{n,2}$be the set of all double 
partitions.  For $\Bla \in \CP_{n,2}$ we put $|\Bla| = n$.  
The Frobenius map $x \mapsto x^q$ acts naturally on the 
multiplicative group $\Bk^*$ of $\Bk$, and we denote by $\Xi_q$ the set of
$F$-orbits in $\Bk^*$.  For each $F$-orbit $\xi \in \Xi_q$, let 
$|\xi|$ be the length of the orbit $\xi$. 
We denote by $\F_{n,q}$ the set of $\CP^{(2)}$-valued functions   
$\vL$ on $\Xi_q$ subject to the condition 
\begin{equation*}
\sum_{\xi \in \Xi_q}|\xi|\cdot |\vL(\xi)| = n.
\end{equation*}
There is an embedding  $\CP_{n,2} \hra \F_{n,q}$ by $\Bla \mapsto \vL$, 
where $\vL(\xi) = \Bla$ if $\xi = \{ e \}\subset \Bk^*$ and 
is equal to zero otherwise. 
We know that the set of $H^F$-orbits in $\CX\uni^F$ is parametrized 
by $\CP_{n,2}$.  We note that 
\par\medskip\noindent
(6.3.1) \ The set of $H^F$-orbits in $\CX^F$ is parametrized by 
$\F_{n,q}$.   
\par\medskip\noindent
In fact the parametrization is done, first  by parameterizing 
semisimple $H^F$-orbits in $(G^{\io\th})^F$, and then  
by parameterizing $Z_H(s)^F$-orbits  
in $(Z_G(s)^{\io\th}\times V)^F\uni$ for a semisimple element 
$s \in (G^{\io\th})^F$ (see 2.2).  Since the parametrization of
semisimple $H^F$-orbits in $(G^{\io\th})^F$ is equivalent to
that of semisimple classes in $GL_n^F$ 
by [BKS, Lemma 2.3.4], (6.3.1) follows.  

\para{6.4.}
For a given pair $(T,\CE)$,  where $T$ is a $\th$-stable maximal torus 
of $G$ contained in a $\th$-stable Borel subgroup, and 
$\CE$ is a tame local system on $T^{\io\th}$, 
we denote by $\wh\CX_{(T,\CE)}$ the set
of isomorphism classes of simple perverse sheaves on $\CX$ 
appearing in the decomposition of $K_{T,\CE}$.  
Recall that the set $\wh\CX$ of (isomorphism classes of) 
character sheaves on $\CX$ is a union of $\wh\CX_{(T,\CE)}$ 
for a various choice of $(T,\CE)$ (see [SS, 4.1]).
We denote by $\wh\CX^F$ the set of character sheaves $A$ such that
$F^*A \simeq A$. For $A \in \wh\CX^F$, we fix an isomorphism 
$\vf_A: F^*A \isom A$, and consider the characteristic function 
$\x_{A, \vf_A}$ on $\CX^F$.  Thus $\x_{A, \vf_A}$ is an $H^F$-invariant 
function on $\CX^F$, which depends only on $A$ up to scalar.
\par
Let $\ZC_q(\CX)$ be the $\Ql$-space of $H^F$-invariant 
functions on $\CX^F$. 
We have the following proposition.

\begin{prop}  
\begin{enumerate}
\item  Assume that $(T,\CE)$ and $(T',\CE')$ are not 
$H$-conjugate.  Then $\wh\CX_{(T,\CE)}$ and 
$\wh\CX_{(T',\CE')}$ are disjoint.
\item
$\{ \x_{A,\vf_A} \mid A \in \wh\CX^F\}$ gives rise to 
a basis of $\ZC_q(\CX)$. 
\end{enumerate}
\end{prop}  

\begin{proof}
Let $\CE$ be a tame local system on $T_0^{\io\th}$ such that 
$\CY_{\CE} \ne \emptyset$. 
One can construct functions $\x_{K_{\r},\vf_{\r}}$ as in 6.1
for each $\r \in (\wt\CW_{\CE})\wg_{\ex}$ 
by decomposing $K_{T_{ww_1}, \CE_{ww_1}}$.
We show that 
\par\medskip\noindent
(6.5.1) \ The functions $\x_{K_{\r},\vf_{\r}}$ , where 
$\CE$ runs over tame local systems on $T^{\io\th}_0$ such that 
$\CY_{\CE} \ne \emptyset$, and $\r \in (\wt\CW_{\CE})\wg_{\ex}$, 
are linearly independent.
\par\medskip\noindent
We define a bilinear form on $\ZC_q(\CX)$ by 
\begin{equation*}
\lp f, h \rp = |H^F|\iv\sum_{z \in \CX^F}f(z)h(z)
\end{equation*}
for $f,h \in \ZC_q(\CX)$.   In order to prove (6.5.1), 
it is enough to show by Proposition 6.2, that the matrix 
$Z = (\lp \x_{K_{\r},\vf_{\r}}, \x_{K_{\r'},\vf_{\r'}}\rp)_{\r,\r'}$
is nondegenerate for a fixed $\CE, \CE'$ such that 
$\CE' = \CE^{\vee}$ 
(here $\r \in (\wt\CW_{\CE})\wg_{\ex}, 
\r' \in (\wt\CW_{\CE'})\wg_{\ex}$).
Note that for $w \in W_n$, $|T_w^{\th, F}|$ is the order of 
an $F$-stable  maximal torus in $Sp_{2n}$ corresponding to 
$w \in W_n$, and $|T_w^{\io\th, F}|$ is the order of an 
$F$-stable  maximal torus $S_{\bar w}$ of $GL_n$, where 
$\bar w$ is the image of $w$ under the map $W_n \to S_n$. Then 
again by Proposition 6.2, one can write
\begin{equation*}
\lp \x_{K_{\r},\vf_{\r}}, \x_{K_{\r'},\vf_{\r'}}\rp
   = R(q)\iv \sum_{w \in \wt\CW_{\CE}}R_w(q)
            \x_{\wt\r}(w\g)\x_{\wt\r'}(w\g'),
\end{equation*}   
where $R(t), R_w(t)$ are polynomials in $t$ such that 
$R_w(t)$ is a monic of the same degree for any $w$.
Hence by the orthogonality relations for extended irreducible characters
of $\wt\CW_{\CE}$, each entry of $R(q) Z$ is a polynomial in $q$, 
the diagonal entries have the same degree, and off diagonal entries
have strictly smaller degrees.  Thus $Z$ is nondegerate and 
(6.5.1) holds.   
\par
Now 
$\CE$ determines a tame local system $\CE_0$ on an $F$-split 
 maximal torus $S_0$ of $GL_n$, 
and so determines an $F$-stable semisimple class $(s_0)$ 
in the dual group $GL_n^* \simeq GL_n$. 
Thus it determines an $F$-stable semisimple class 
$(s)$ in $G^{\io\th}$.  The class $(s)$ has the property that 
$\wt\CW_{\CE}$ is isomorphic to the Weyl group of 
$Z_H(s)$.  It follows that the cardinality of functions in 
(6.5.1) coincides with the cardinality of $\F_{n,q}$.  
This implies that the functions in (6.5.1) gives a basis of 
$\ZC_q(\CX)$.  If the pair $(T,\CE)$ is $F$-stable, it determines 
a unique $\CE_0$ on $T_0^{\io\th}$ such that $\CY_{\CE_0} \ne \emptyset$, 
and this give a bijection between $\wh\CX^F_{(T,\CE)}$ and the set 
in (6.5.1) corresponding to $\CE_0$.  Hence (ii) and (i) follows 
from this. 

\end{proof}

\par\vspace{1cm}
\noindent
T. Shoji \\
Graduate School of Mathematics, Nagoya University \\ 
Chikusa-ku, Nagoya 464-8602, Japan  \\
E-mail: \verb|shoji@math.nagoya-u.ac.jp|
\par\bigskip\bigskip\noindent
K. Sorlin \\
L.A.M.F.A, CNRS UMR 7352, 
Universit\'e de Picardie-Jules Verne \\
33 rue Saint Leu, F-80039, Amiens, Cedex 1, France \\
E-mail : \verb|karine.sorlin@u-picardie.fr|

 \end{document}